\pdfoutput=1
\documentclass[12pt,twoside,a4paper]{article}

\usepackage{fancyhdr}
\usepackage{titlesec}
\usepackage{titletoc}
\usepackage{graphicx}

\usepackage{amsmath}
\usepackage{amstext}
\usepackage{amsbsy}
\usepackage{amssymb}
\usepackage{theorem}
\usepackage{dsfont}
\usepackage{times}
\usepackage[T1]{fontenc}
\usepackage{eucal}
\usepackage{mathrsfs}
\usepackage{url}
\renewcommand{\today}{September 29, 2009}
\newcommand{\ds}{\displaystyle}
\newcommand{\ts}{\textstyle}
\newcommand{\tint}{{\ts \int}}

\renewcommand{\mathbb}{\mathds}

\newcommand{\C}[1]{\mathcal{#1}}

\newcommand{\ov}[1]{\overline{#1}}
\newcommand{\wt}[1]{\widetilde{#1}}
\newcommand{\wh}[1]{\widehat{#1}}
\newcommand{\B}[1]{\mathds{#1}}

\numberwithin{equation}{section}
\newtheorem{thm}{Theorem}[section]
\newtheorem{prop}[thm]{Proposition}
\newtheorem{proposition}[thm]{Proposition}
\newtheorem{lemma}[thm]{Lemma}

{\theorembodyfont{\rm} 

}
\theoremstyle{remark}

\theoremheaderfont{\sc}
\newenvironment{proof}{{\sc Proof.}}{\ $\square$}
\titleformat{\section}[block]{\sc\center}{\thetitle.}{5pt}{}[]
\titlespacing{\section}{0pt}{*4.5}{*3}
\titleformat{\subsection}[runin]{\sc}{\thetitle.}{5pt}{}[.]
\titlespacing{\subsection}{0pt}{*3}{*2}
\titleformat{\subsubsection}[runin]{\it}{\thetitle.}{5pt}{}[.]
\titlespacing{\subsubsection}{0pt}{*2}{*2}
\contentsmargin[30pt]{30pt}
\titlecontents{section}[15pt]{\vspace{6pt}\sc}
{\contentslabel[\thecontentslabel.]{15pt}}{\hspace*{-15pt}}
{\titlerule*[1pc]{.}\contentspage[\rm\thecontentspage]}
\titlecontents{subsection}[37pt]{\vspace{3pt}\small\sc}
{\contentslabel[\thecontentslabel.]{22pt}}{\hspace*{-22pt}}
{\titlerule*[1pc]{.}\contentspage[\normalsize\rm\thecontentspage]}
\titlecontents{subsubsection}[69pt]{\it}
{\contentslabel[\thecontentslabel.]{32pt}}{\hspace*{-32pt}}
{~\titlerule*[1pc]{.}\contentspage[\rm\thecontentspage]}

\newcommand{\thmref}[1]{\ref{#1} (page \pageref{#1})}
\newcommand{\myeq}[1]{{\rm (\ref{#1}, page \pageref{#1})}}
\allowdisplaybreaks

\DeclareMathOperator{\XX}{\text{\bf \textsf X}}

\begin{document}
\pagestyle{fancy}
\renewcommand{\sectionmark}[1]{\markboth{#1}{}}
\renewcommand{\subsectionmark}[1]{\markright{#1}}
\fancyhf{}
\fancyhead[LE]{\small\sc\nouppercase{\leftmark}}
\fancyhead[RO]{\small\sc\nouppercase{\rightmark}}
%\fancyhead[LE,RO]{\thepage}
\fancyfoot[LO,RE]{\small\sc Olivier Catoni}
\fancyfoot[RO,LE]{\thepage}
\renewcommand{\footruleskip}{1pt}
\renewcommand{\footrulewidth}{0.4pt}
\newcommand{\mypoint}{\makebox[1ex][r]{.\:\hspace*{1ex}}}
\addtolength{\footskip}{11pt}
%\pagestyle{headings}
%\frontmatter
\pagestyle{plain}
\renewcommand{\thefootnote}{}
\begin{center}
{\bf High confidence estimates of the mean\\ of heavy-tailed real
random variables}\\[12pt]
{\sc Olivier Catoni}\\[12pt]
{\small \it \today }\\[12pt]
\end{center}
\footnotetext{CNRS --  UMR 8553,
D\'epartement de Math\'ematiques et Applications,
Ecole Normale Sup\'erieure, 45, rue d'Ulm, F75230 Paris 
cedex 05, and INRIA Paris-Rocquencourt -- CLASSIC team.}
%\maketitle
%\begin{center}
%\begin{minipage}{0.9\textwidth}

{\small
{\sc Abstract :} 
We present new estimators of the mean of a real valued random 
variable, based on PAC-Bayesian iterative truncation.
We analyze the non-asymptotic minimax properties of 
the deviations of estimators for distributions having
either a bounded variance or a bounded kurtosis. 
It turns out that these minimax deviations are of 
the same order as the deviations of the empirical mean 
estimator of a Gaussian distribution. Nevertheless, 
the empirical mean itself performs poorly at 
high confidence levels for the worst distribution 
with a given variance or kurtosis (which turns
out to be heavy tailed). To obtain (nearly) minimax 
deviations in these broad class of distributions, 
it is necessary to use some more robust estimator, 
and we describe an iterated truncation scheme whose
deviations are close to minimax. In order to calibrate
the truncation and obtain explicit confidence intervals, 
it is necessary to dispose of a prior bound either
on the variance or the kurtosis. When a prior bound 
on the kurtosis is available, we obtain as a by-product
a new variance estimator with good large deviation properties.
When no prior bound is available, it is still possible 
to use Lepski's approach to adapt to the unknown variance, 
although it is no more possible to obtain observable confidence
intervals.  \\[1ex]
}
%\tableofcontents
%\clearpage
\pagestyle{fancy}
%\mainmatter
%{\fontencoding{U}\fontfamily{ygoth}\selectfont}

\newcommand\eqdef{\triangleq}
\newcommand\A{\mathcal{A}}
\newcommand\cB{\mathcal{B}}
\newcommand\cC{\Theta}
\newcommand\E{\mathbb{E}}
\newcommand\cE{\mathcal{E}}
\newcommand\cF{\mathcal{F}}
\newcommand\cH{\mathcal{H}}
\newcommand\I{\mathcal{I}}
\newcommand\J{\mathcal{J}}
\newcommand\cK{\mathcal{K}}
\newcommand\cL{\mathcal{L}}
\newcommand\M{\mathcal{M}}
\newcommand\N{\mathcal{N}}
\renewcommand\P{\mathbb{P}}
\newcommand\R{\mathbb{R}}
\renewcommand\S{\mathcal{S}}
\newcommand\bcR{B} %{\bar{\mathcal{R}}}
\newcommand\cR{\tilde{B}} %\mathcal{R}}
\newcommand\W{\mathcal{W}}
\newcommand\X{\mathcal{X}}
\newcommand\Y{\mathcal{Y}}
\newcommand\Z{\mathcal{Z}}

\newcommand\jyem{\em}

\newcommand\hdelta{\hat{\delta}}
\newcommand\hth{\hat{\theta}}
\newcommand\tth{\tilde{\theta}}
\newcommand\hlam{\hat{\lam}}
\newcommand\lamerm{\hlam^{\textnormal{(erm)}}}

\newcommand\tf{\tilde{f}}
\newcommand\tlam{\tilde{\lam}}

\newcommand\be{\beta}
\newcommand\ga{\gamma}
\newcommand\kap{\kappa}
\newcommand\lam{\lambda}

\newcommand\sint{{\textstyle \int}}
\newcommand\inth{\sint \pi(d\theta)}
\newcommand\inthj{\sint \pij(d\theta)}
\newcommand\inthp{\sint \pi(d\theta')}

\newcommand\gas{\ga^*}

\newcommand\hpi{\hat{\pi}}
\newcommand\pis{\pi^*}
\newcommand\tpi{\tilde{\pi}}

\newcommand\ovth{\ov{\theta}}
\newcommand\wth{\wt{\theta}}
\newcommand\wthj{\wt{\theta}_j}

\newcommand\eps{\varepsilon}
\newcommand\logeps{\log(\eps^{-1})}
\newcommand\leps{\log^2(\eps^{-1})}

\newcommand{\bi}{\begin{itemize}}
\newcommand{\ei}{\end{itemize}}

\newcommand\ovR{\ov{R}}

\newcommand\hhpi{\hat{\hat{\pi}}}
\newcommand\wwth{\wt{\wt{\theta}}}
\newcommand\pia{\pi^{(1)}}
\newcommand\pib{\pi^{(2)}}
\newcommand\pij{\pi^{(j)}}
\newcommand\tpia{\tilde{\pi}^{(1)}}
\newcommand\tpib{\tilde{\pi}^{(2)}}
\newcommand\tpij{\tilde{\pi}^{(j)}}
\newcommand\wtha{\wt{\theta}_1}
\newcommand\wthb{\wt{\theta}_2}
\newcommand\wths{\wt{\theta}_s}
\newcommand\wtht{\wt{\theta}_t}

\newcommand\cmin{c_{\min}}
\newcommand\cmax{c_{\max}}

\newcommand\ra{\rightarrow}

\newcommand\hatt{\hat{t}}
\newcommand\argmax{\textnormal{argmax}}
\newcommand\argmin{\textnormal{argmin}}

\newcommand\diag{\textnormal{Diag}}

\newcommand\Fg{\cF^\#}

\newcommand\vp{\varphi}
\newcommand\tvp{\tilde{\vp}}
\newcommand\tphi{\tilde{\phi}}
\renewcommand\th{\theta}

\newcommand\vsp{\vspace{1cm}}
\newcommand\lhs{\text{l.h.s.}}
\newcommand\rhs{\text{r.h.s.}}

\newcommand\expe[2]{\undc{\E}{#1\sim#2}}
\newcommand\expec[2]{\undc{\E}{#1\sim#2}}
\newcommand\expecc[2]{\E_{#1}}
\newcommand\expecd[2]{\E_{#2}}

\newcommand\hC{\hat{C}}
\newcommand\hf{\hat{f}}
\newcommand\hrho{\hat{\rho}}

\newcommand\br{\bar{r}}
\newcommand\chr{\check{r}}
\newcommand\bR{\bar{R}}

\newcommand\lan{\langle}
\newcommand\ran{\rangle}

\newcommand\logepsg{\log(|\Cg|\eps^{-1})}

\newcommand\Pemp{\hat{\P}}

\newcommand\fracl[2]{{(#1)}/{#2}}
\newcommand\fracc[2]{{#1}/{#2}}
\newcommand\fracr[2]{{#1}/{(#2)}}
\newcommand\fracb[2]{{(#1)}/{(#2)}}

\newcommand\hfproj{\hf^{\textnormal{(proj)}}}
\newcommand\thproj{\hat{\th}^{\textnormal{(proj)}}}

\newcommand\hfols{\hf^{\textnormal{(ols)}}}
\newcommand\thfols{\tilde{f}^{\textnormal{(ols)}}}
\newcommand\thols{\hat{\th}^{\textnormal{(ols)}}}
\newcommand\therm{\hat{\th}^{\textnormal{(erm)}}}
\newcommand\hferm{\hf^{\textnormal{(erm)}}}
\newcommand\zols{\zeta^{\textnormal{(ols)}}}

\newcommand\thrid{\tilde{\th}} %\th^*_{\textnormal{(ridge)}}}
\newcommand\frid{\tilde{f}} %f^*_{\textnormal{(ridge)}}}
\newcommand\freg{f^{\textnormal{(reg)}}}
\newcommand\hfrlam{\hf^{\textnormal{(ridge)}}} %_{\lam}
\newcommand\hfllam{\hf^{\textnormal{(lasso)}}} %_{\lam}

\newcommand\flin{f^*_{\textnormal{lin}}}
\newcommand\thlin{\th^{\textnormal{(lin)}}}
\newcommand\Flin{\mathcal{F}_{\textnormal{lin}}}

\renewcommand\Phi{\XX}
\newcommand\demi{\frac{1}{2}}
\newcommand\demic{\fracc{1}{2}}

\newcommand\substa[2]{\substack{#1\\#2}}
\newcommand\substac[2]{{#1\,;\,#2}}
\newcommand\tpsi{\tilde{\psi}}
\newcommand\tzeta{\tilde{\zeta}}
\newcommand\ta{\tilde{a}}
\newcommand\chis{\chi_\sigma}
\newcommand\tchi{\tilde{\chi}}
\newcommand\tchis{\tchi_\sigma}
\newcommand\psis{\psi_\sigma}

\newcommand\tA{\tilde{A}}
\newcommand\tL{\tilde{L}}

\newcommand\hL{\hat{L}}
\newcommand\hcE{\hat{\cE}}
\newcommand\hDe{\hat{\cE}}
\newcommand\cEb{\cE^{\sharp}}
\newcommand\La{L^{\flat}}
\newcommand\cEa{\cE^{\flat}}
\newcommand\Lb{L^{\sharp}}

\newcommand\Pa{P^{\flat}}
\newcommand\Pb{P^{\sharp}}

\newcommand\ela{\tilde{\ell}}
\newcommand\hpig{\hpi^{\textnormal{(Gibbs)}}}

\newcommand\sigmb{\phi}
\newcommand{\tR}{\tilde{\cR}}
\newcommand{\logdeps}{\log(4d\eps^{-1})}
\newcommand{\logddeps}{\log(2d^2\eps^{-1})}

\section*{Introduction}
This paper is devoted to the estimation of the mean of a real random variable
from an independent identically distributed sample. We will emphasize the 
following issues : 
\begin{itemize}
\item obtaining non asymptotic confidence intervals;
\item getting high confidence levels; 
\item proving nearly minimax bounds in the class of 
distributions with a bounded variance and in the class of distributions 
with a bounded kurtosis. 
\end{itemize}
To achieve these goals, we combine two kinds of tools: truncated estimates and PAC-Bayesian theorems (
\cite{McA98,McA99,McA01,Cat05,Aud03b,Alq08}).

The general conclusion is that the empirical mean estimate behaves poorly at high confidence levels and that the worst case is reached for heavy tailed distributions, as the proofs of the lower bounds show.

This is the bad news. The good news is that, using iterated truncation schemes, 
it is possible to recover confidence intervals whose widths are close
to the (optimal) width of the confidence interval of the empirical mean of 
a Gaussian distribution, even at very high confidence levels. From a technical point of view, it is possible to build an estimator with an exponential tail even when the sample distribution has only a finite variance. This came out 
to us as a surprise while working on the more elaborate topic of 
regression estimation \cite{Cat09}, and gave us the spur to work out the estimation 
of the mean in details, this simpler case lending itself to tighter
computations. 

The weakest hypothesis we will consider is the existence of 
a finite variance. While it is possible to adapt a truncation scheme
when the variance is unknown, using Lepski's approach, 
some more information is required to compute an observable confidence interval. 
We study two situations: the case when the variance or some upper bound 
is known and the case when the kurtosis or some upper bound is known. 
In order to assess the quality of the results, we prove corresponding
lower bounds for the best estimator of the worst distribution (following thus the minimax approach), 
and for the empirical mean estimate of the worst distribution, 
to assess the improvement brought 
by the PAC-Bayesian truncation scheme. We plot the numerical values 
of these upper and lower bounds for typical finite 
sample sizes to show the gap between them.

Let us end this introduction with a few words in favour of high confidence
levels. One reason to seek them is when the estimated quantity is critical
from a safety or economical point of view. We will not elaborate on this. 
Another setting where high confidence levels are required is when 
lots of estimates are to be computed and compared in some statistical 
learning scenario. Let us imagine, for instance, that some parameter
$\theta \in \Theta$ is to be tuned in order to optimize the answer
of some loss function $f_{\theta}$ to some random input $X$. Let us 
consider a split sample scheme where two i.i.d. samples $X_1, \dots, 
X_s \overset{\text{def}}{=} X_1^s$ and $X_{s+1}, \dots, X_{s+n} 
\overset{\text{def}}{=} X_{s+1}^{s+n}$ are used, one to build some 
estimators $\wh{\theta}_k(X_1^s)$ of $\argmin_{\theta \in \Theta_k} \B{E} 
\bigl[ f_{\theta}(X) \bigr]$ in subsets $\Theta_k$, $k=1, \dots, K$
of $\Theta$, and the other to test those estimators and keep hopefully
the best. This is a very common model selection situation. 
One can think for instance of the choice of a basis to expand 
some regression function. If $K$ is large, estimates of 
$\B{E} \bigl[ f_{\wh{\theta}_k(X_1^s)}(X_{s+1}) \bigr]$ will be required
for a lot of values of $k$. In order to keep safe from over-fitting,
very high confidence levels will be required if the resulting
confidence level is to be computed through a union bound (because
no special structure of the problem can be used to do better).
Namely, a confidence level of $1 - \epsilon$ on the final result
of the optimization on the test sample will require a confidence
level of $1 - \epsilon/K$ for each mean estimate on the test sample. 
Even if $\epsilon$ is not very small (like, say, $5/100$), 
$\epsilon/K$ may be very small. For instance, if 10 parameters
are to be selected among a set of 100, this gives $K = 
{100 \choose 10} \simeq 1.7 \cdot 10^{13}$. 
In practice some heuristic scheme will be used 
to compute only a limited number of estimators $\wh{\theta}_k$, like adding parameters one at a time, 
 choosing
at each step the one with the best estimated performance increase (in our example, this requires to compute $1000$ estimators instead of 
${100 \choose 10}$). Nonetheless,
asserting the quality of the resulting choice 
requires a union 
bound on the whole set of possible outcomes of the 
data driven heuristic, and therefore calls for  
very high confidence levels for each estimate of the mean performance 
$\B{E} \bigl[ f_{\wh{\theta}_k (X_1^s)}(X_{s+1}) \bigr]$ 
on the test set. 

Our study has several reasons to recommend itself as addressing 
the question of robust statistics: we prove distribution free bounds, 
truncation operates mainly on outliers and the lower bounds 
show that the worst behaviour of the empirical mean is achieved
on heavy tailed distributions. Anyhow, our point of view is quite
different from the classical setting of robust statistics, as 
epitomised by Peter Huber \cite{Huber}. Indeed, our framework
is not perturbative, --- we do not assume the sample to be drawn 
from a mixture of known and unknown distributions ---, and is not 
asymptotic either, since it is based on finite sample exponential
inequalities for suitable auxiliary variables. 
Moreover Huber's approach is a minimax study of the 
variance of estimators, whereas we analyze the minimax
properties of their deviations.
From a more practical point of view, the fact that the empirical 
mean is unstable is well known, and any statistical package
provides tools to deal with outliers. It is interesting though
that it shows in the equations, even when no sample contamination is 
assumed, by simply considering a minimax setting on a broad 
set of distributions including heavy tailed ones, and looking
at the deviations of estimators, rather than 
focussing on their variance. 

\section{Some truncated mean estimate}

Let $(Y_i)_{i=1}^n$ be an i.i.d. sample drawn from some unknown probability 
distribution $\B{P}$ on the real line $\B{R}$ equipped with the
Borel $\sigma$-algebra $\C{B}$. Let $Y$ be independent from $(Y_i)_{i=1}^n$ 
with the same marginal distribution $\B{P}$. 
Let $m$ be the mean of $Y$ and let $v$ be its variance: 
$$ 
\B{E}(Y) = m \quad \text{ and } \quad \B{E}[ (Y - m)^2 ] = v.
$$

Starting from some initial guess $\theta_0$ about the value 
of the mean, prior to any observation, let us consider the 
thresholded estimator
\begin{equation}
\label{eq1.1.2}
\wh{\theta}_{\alpha}(\theta_0)  = \theta_0 + \frac{1}{n \alpha} 
\sum_{i=1}^n T \bigl[ \alpha (Y_i - \theta_0) \bigr], 
\end{equation}
where the threshold function $T$ is defined as
$$
T(x) = \frac{1}{2} \log \left( \frac{\ds 1 + x + \frac{x^2}{2}}{
\ds 1 - x + \frac{x^2}{2}} \right).
$$
\mbox{} \hfill \includegraphics{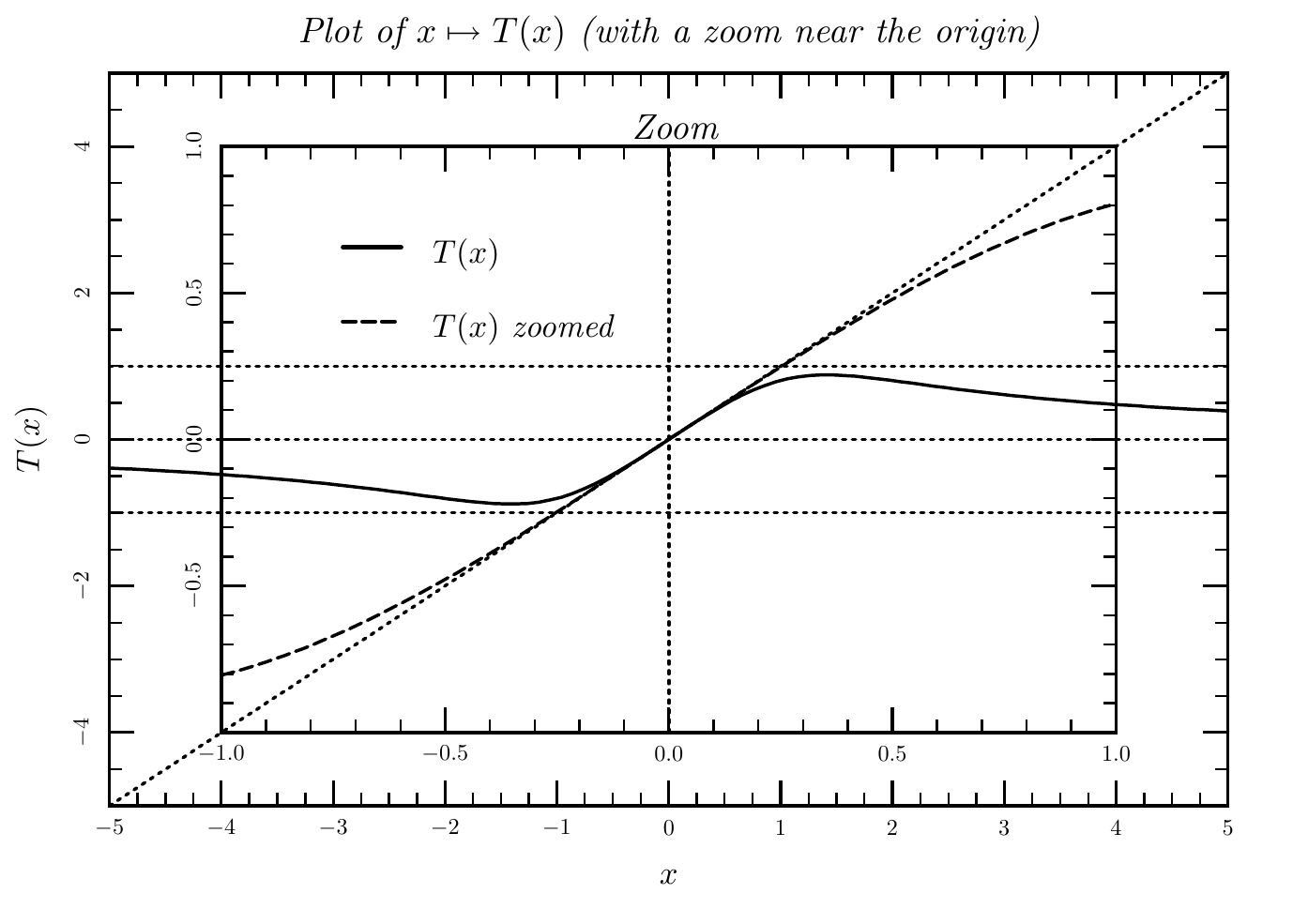}   
\hfill \mbox{}\\[-8ex]
\begin{prop}
\label{prop3.1}
Assume that $v \leq v_0$ and $\lvert \theta_0 - m \rvert 
\leq \delta_0$, where $v_0$ and $\delta_0$ are known prior bounds.
With probability at least $1 - 2 \epsilon$, 
$$
\lvert \wh{\theta}_{\alpha}(\theta_0)  - m \rvert \leq \frac{\alpha v_0}{2} 
+ \frac{\log(\epsilon^{-1})}{n \alpha} 
+ \frac{\alpha^2 \delta_0}{2} \bigl( 1 + \alpha \delta_0 
\bigr) \biggl( \frac{\delta_0^2}{3} + v_0 \biggr).
$$

Choosing $\ds \alpha = \sqrt{\frac{2 \log(\epsilon^{-1})}{nv_0}}$, we 
get, with probability $1 - 2 \epsilon$, 
$$
\lvert \wh{\theta}_{\alpha}(\theta_0)  - m \rvert 
\leq \sqrt{\frac{2 v_0 \log(\epsilon^{-1})}{n}} 
+ \frac{\log(\epsilon^{-1}) \delta_0}{3 n v_0} \bigl( 
\delta_0^2 + 3 v_0 \bigr) \biggl( 1 + \delta_0 
\sqrt{\frac{2 \log(\epsilon^{-1})}{nv_0}} \biggr).
$$
Choosing 
$\ds \alpha = \sqrt{\frac{2}{nv_0}}$ independently of $\epsilon$
we get with  
probability at least $1 - 2 \epsilon$,
$$
\lvert \wh{\theta}_{\alpha}(\theta_0) - m \rvert \leq  
\bigl[ 1 + \log(\epsilon^{-1}) \bigr] 
\sqrt{\frac{v_0}{2n}} + \frac{\delta_0}{3 n v_0} \bigl(\delta_0^2 
+ 3v_0 \bigr) \biggl( 1 + \delta_0 \sqrt{\frac{2}{nv_0}} \biggr).
$$
\end{prop}
\begin{prop}
\label{prop3.2}
Assume that $v \leq v_0$ and $\lvert m - \theta_0 \rvert \leq \delta_0$, 
where $v_0$ and $\delta_0$ are known prior bounds. 
With probability at least $1 - 2 \epsilon$, 
$$
\lvert \wh{\theta}_{\alpha}(\theta_0) - m \rvert \leq \frac{\alpha(v_0 + \delta_0^2)}{2} 
+ \frac{\log(\epsilon^{-1})}{n \alpha}.
$$
Choosing $\ds \alpha = \sqrt{\frac{2 \log(\epsilon^{-1})}{
n (v_0 + \delta_0^2)}}$, we get
$$
\lvert \wh{\theta}_{\alpha}(\theta_0) - m \rvert \leq 
\sqrt{\frac{2 (v_0 + \delta_0^2) \log(\epsilon^{-1})}{n}}.
$$
\end{prop}

Let us remark that the estimates proved here are valid for any confidence
level $1 - 2 \epsilon$. In particular, when $\wh{\theta}_{\alpha}(\theta_0)$ is 
independent of $\epsilon$, it has a subexponential tail distribution, 
even in the case when $\B{P}$ has not. The proofs are gathered in the last section 
of the paper.

\section{Iterated mean estimates}

The width of the confidence intervals proved in the previous section 
depends heavily on the value of the prior bound $\delta_0$. 
On the other hand, they lend themselves naturally to an 
iterated scheme. Here we will iterate Proposition \thmref{prop3.2}, 
where the dependence of the bound on $\delta_0$ is the best.

\begin{prop}
\label{prop2.1}
Let us assume that $v \leq v_0$ and $\lvert m - \theta_0 \rvert 
\leq \delta_0$, where $v_0$, $\theta_0$ and $\delta_0$ are known 
prior to observing the sample.
Let $U_i$, $i=2, \dots, k$ be uniform real random variables 
in the interval $(-1, +1)$, independent of each other and of 
everything else. 
Let us define 
\begin{align*}
\delta_1 & = \sqrt{\frac{2 (v_0 + \delta_0^2) \log(\epsilon_1^{-1})}{n}}, \\
\alpha_1 & = \sqrt{\frac{2 \log(\epsilon_1^{-1})}{n(v_0 + \delta_0^2)}}, \\
\wt{\theta}_1 & = \wh{\theta}_{\alpha_1}(\theta_0), \\
& \vdots\\
\gamma_i & = \log ( 1 + x_i^{-1}),\\
\delta_{i} & = \sqrt{\frac{2 \bigl[ v_0 +(1 + x_i)^2 \delta_{i-1}^2 \bigr] 
\bigl[ \log(\epsilon_i^{-1}) + \gamma_{i} \bigr] }{n}},\\
\alpha_i & = \sqrt{\frac{2 \bigl[ \log(\epsilon_i^{-1}) + \gamma_{i} \bigr]}{
n \bigl[ v_0 + (1 + x_i)^2 \delta_{i-1}^2 \bigr] }},\\
\wt{\theta}_i & = \wh{\theta}_{\alpha_i}(\wt{\theta}_{i-1} + 
x_i \delta_{i-1} U_i),\\
& \vdots
\end{align*}
With probability at least $\ds 1 - 2 \sum_{i=1}^k \epsilon_i$, 
$$
\lvert m - \wt{\theta}_k \rvert \leq \delta_k.
$$
\end{prop}

Let us see how it behaves, choosing $x_i = 1/10$ and 
$\epsilon_1 = \dots = \epsilon_{k-1} = (\epsilon - 
\epsilon_k)/(k-1) = \epsilon/10$. The following two plots
of $\delta_k$ against $\epsilon$ show that this iterated estimate permit very large
values of the prior bound $\delta_0$, without 
any substantial loss of accuracy (for a suitable number of iterations).\\
\mbox{} \hfill \includegraphics{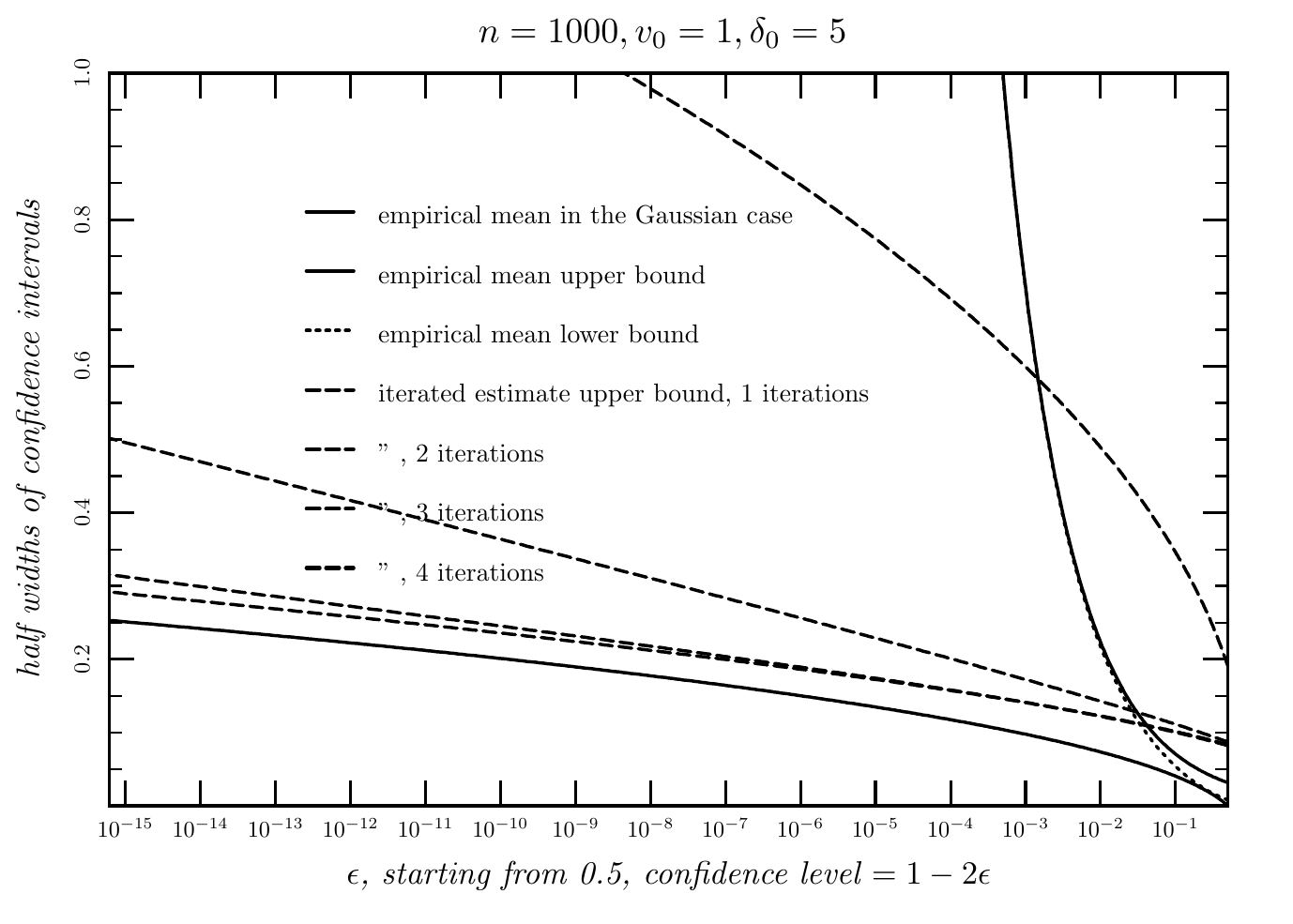}   
\hfill \mbox{}\\
\mbox{} \hfill \includegraphics{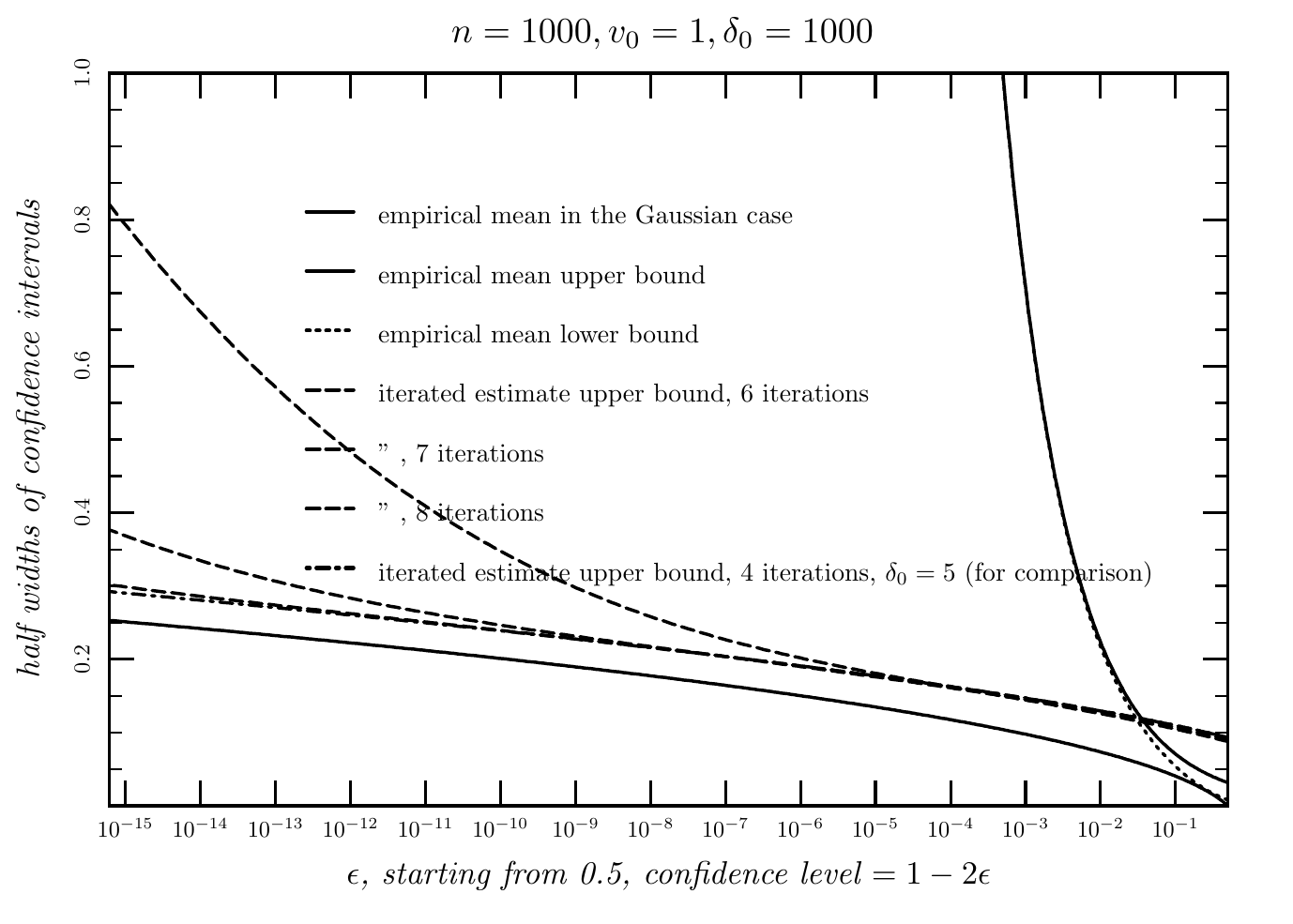}   
\hfill \mbox{}\\[-8ex]
\section{With no prior knowledge of the mean}

In the case when the prior bound $\delta_0$ is not available, we can 
modify the iterated scheme of the previous section, using the empirical 
mean estimator as a first step.

\begin{prop}
\label{prop3.1.2}
Let us assume that $v \leq v_0$, where $v_0$ is a known prior 
bound. Let $U_i$, $i = 2, \dots, k$ be uniform real random variables
in the interval $(-1, +1)$, independent of each other and of everything 
else. Let us define 
\begin{align*}
\delta_1 & = \sqrt{\frac{v_0}{2 n \epsilon_1}},\\
\wt{\theta}_1 & = \frac{1}{n} \sum_{i=1}^n Y_i,\\
& \vdots\\
\gamma_i & = \log( 1 + x_i^{-1}),\\
\delta_i & = \sqrt{\frac{2 \bigl[ v_0 + (1 + x_i)^2 \delta_{i-1}^2 
\bigr] \bigl[ \log (\epsilon_i^{-1}) + \gamma_i \bigr]}{n}},\\
\alpha_i & = \sqrt{ \frac{2 \bigl[ \log(\epsilon_i^{-1}) + \gamma_i \bigr]}{
n \bigl[ v_0 + (1 + x_i)^2 \delta_{i-1}^2 \bigr]}},\\
\wt{\theta}_i & = \wh{\theta}_{\alpha_i} 
\bigl( \wt{\theta}_{i-1} + x_i \delta_{i-1} U_i \bigr),\\
& \vdots
\end{align*}
With probability at least $1 - 2 \sum_{i=1}^k \epsilon_i$, 
$$
\bigl\lvert m - \wt{\theta}_k \bigr\rvert \leq \delta_k.
$$
\end{prop}

So in this iterated scheme, we start from the empirical mean 
estimator and improve it gradually. We will show later
that the confidence interval used here for the empirical 
mean is close to optimal in the worst case. What the next 
plot shows therefore, is that the iterated estimate
brings a huge improvement for high confidence levels, 
allowing to stay close to the deviations of the empirical mean 
of a Gaussian distribution for confidence levels virtually as 
high as wished: this iterative truncation scheme behaves almost 
as the empirical mean estimate of a Gaussian distribution 
would behave, for any distribution with a known finite variance, 
and beats the empirical mean in the worst case for confidence levels
starting from around $94 \%$ for a sample of size $1000$.\\
\mbox{} \hfill \includegraphics{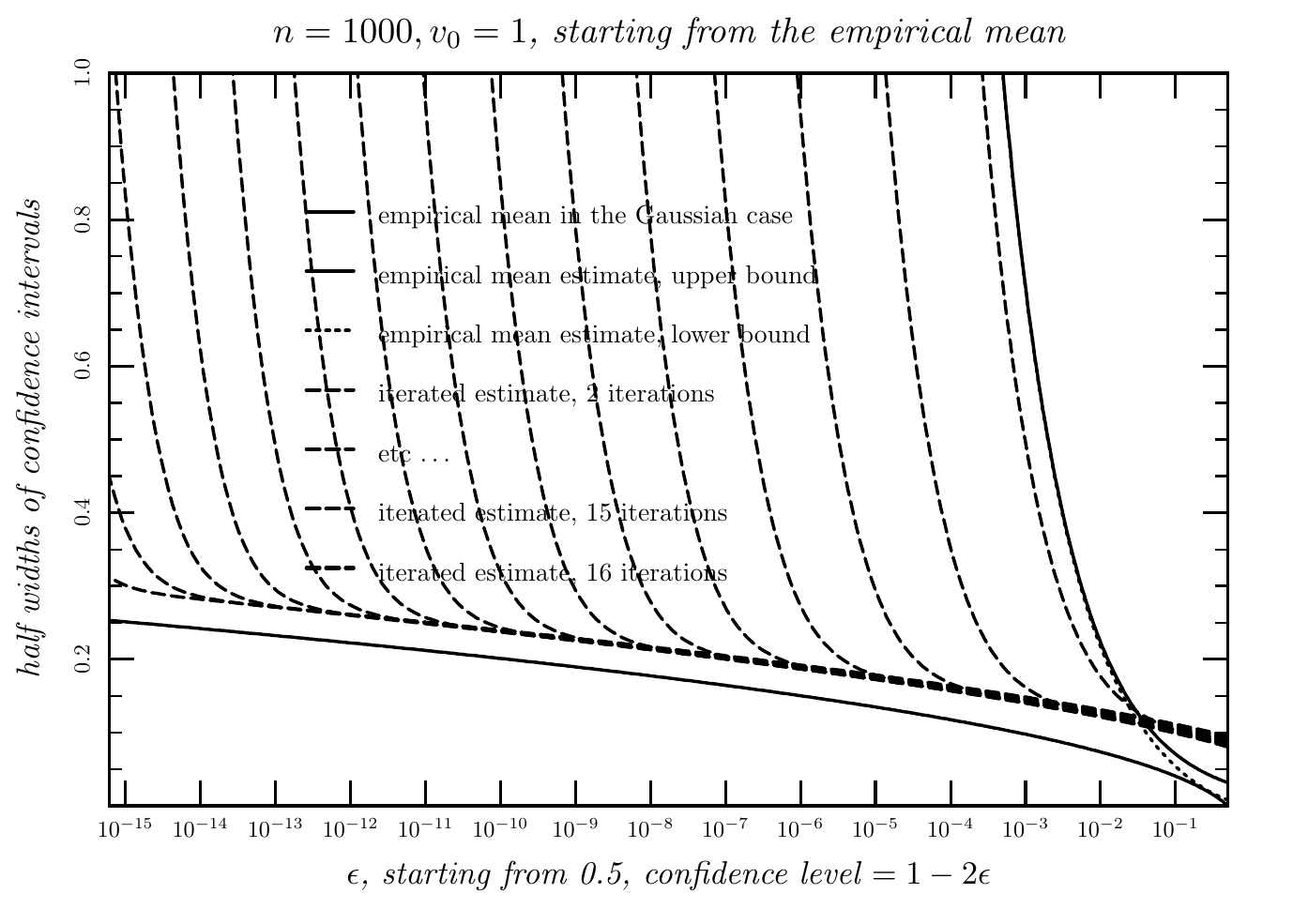}   
\hfill \mbox{}\\[-8ex]
\section{Last step improvement}

In this section, we introduce a more elaborate estimate
to perform the last step of the iteration.
The result of the previous steps will be described as $\wt{\theta}_1$, 
assumed to be some mean estimator
satisfying with probability at least $1 - 2 \epsilon_1$
\begin{equation}
\label{eq4.1}
\lvert m - \wt{\theta}_1 \rvert \leq \delta_1.
\end{equation}

Let us consider, for any $\theta_0 \in \B{R}$,  
the Gaussian distribution on the real line with variance 
$(n \beta  \alpha^2)^{-1}$ 
and mean $\theta_0$,  where $\alpha$ and $\beta$ are 
positive parameters to be chosen later:
$$
\rho_{\theta_0}(d \theta) = \sqrt{\frac{n \beta \alpha^2}{2 \pi}} 
\exp \biggl[ - \frac{n \beta \alpha^2}{2} \bigl(\theta - \theta_0\bigr)^2 \biggr] 
d \theta.
$$
This will serve to define some truncated mean estimate
\begin{multline}
\label{eq1.1.4}
M_{\alpha}(\theta_0) 
= \frac{1}{n \alpha} \sum_{i=1}^n \log \biggl\{ 
\tint \rho_{\theta_0}(d \theta) \biggl[ 
1 + \alpha( Y_i - \theta_0 ) + \frac{\alpha^2}{2} \bigl( 
Y_i - \theta_0 \bigr)^2 \biggr] \biggr\}
\\ = \frac{1}{n \alpha} \sum_{i=1}^n 
\log \biggl\{ 1 + \alpha(Y_i - \theta_0) + \frac{\alpha^2}{2} 
\bigl(Y_i - \theta_0 \bigr)^2 + \frac{1}{2 \beta n} \biggr\}. 
\end{multline}
\begin{proposition}
\label{prop2.1.1}
With probability at least $1 - \epsilon$, for any $\theta_0 \in \B{R}$, 
\begin{multline*}
-n \alpha M_{\alpha}(\theta_0) \leq 
n \alpha(\theta_0 - m) \\ + \frac{n \alpha^2}{2} \Bigl[ (\theta_0 - m)^2 
+ v \Bigr] + \frac{1}{2 \beta} + \frac{n \beta \alpha^2}{2} (\theta_0 - m)^2 
- \log(\epsilon).
\end{multline*}
\end{proposition}
Let us insist on the fact that this result holds with probability 
$1 - 2 \epsilon$ {\em uniformly} with respect 
to $\theta_0$, which may therefore be a random variable --- such as
$\wt{\theta}_1$ --- if required.
\begin{prop}
\label{prop1.2}
For any $\theta_0 \in \B{R}$, with probability at least $1 - \epsilon$, 
$$
n \alpha M_{\alpha}(\theta_0) 
\leq - n \alpha (\theta_0 - m) + \frac{n \alpha^2}{2} 
\biggl[ (\theta_0 - m)^2 + v \biggr] + \frac{1}{2\beta}  
- \log(\epsilon).
$$
\end{prop}

\begin{prop}
\label{prop1.4}
Let $v \leq v_0$, where $v_0$ is some known prior bound.
Let $\wt{\theta}_1$ be some estimator satisfying equation \eqref{eq4.1}
with probability at least $1 - 2 \epsilon_1$.
Let us consider the estimator
\begin{equation}
\label{eq1.1.3}
\wh{\theta}_{\alpha} = \inf \Bigl\{ \theta \geq \wt{\theta}_1 - 
\delta_1 : M_{\alpha}(\theta) \leq 0 \Bigr\}.
\end{equation}
Let us define the ancillary function 
\begin{equation}
\label{ancil}
\varphi(x) = 
\begin{cases} \ds \frac{2 x}{1 + \sqrt{1 - 4x}} \leq \frac{x}{1 - 2x}, & 
x \leq 1/4,\\
\ds + \infty, & \text{otherwise}.
\end{cases}
\end{equation}
For any real positive constants $\epsilon_2$ and $\alpha$ such that    
\begin{multline*}
4 n \alpha \delta_1 \leq \bigl[ n \alpha^2 v + \beta^{-1} 
+ 2 \log(\epsilon_2^{-1}) \bigr]  \\ \qquad  \times \varphi \Biggl( 
\frac{(1 + \beta) \bigl[ n \alpha^2 v + \beta^{-1} + 2 \log(\epsilon_2^{-1})
\bigr]}{4n} \Biggr)^{-1}, 
\end{multline*}
which is the case at least when 
\begin{equation}
\epsilon_2 \geq \exp \Biggl\{ 
- n \biggl[\frac{1}{1+\beta} - \alpha \delta_1 
- \frac{\bigl(n \alpha^2 v + \beta^{-1} \bigr)}{2n}
\biggr] \Biggr\}, \label{eq1.1}
\end{equation}
with probability at least 
$1 - 2 \epsilon_1 - 2 \epsilon_2$, 
$$
\lvert m - \wh{\theta}_{\alpha} \rvert \leq 
\frac{2}{(1+\beta) \alpha} 
\varphi \Biggl( \frac{(1+\beta)
\bigl[ n \alpha^2 v + \beta^{-1} - 2 \log(\epsilon_2) \bigr]}{
4 n} \Biggr).
$$
Considering $\ds \alpha = 
\left(\frac{\beta^{-1} - 2 \log(\epsilon_2)}{nv_0}\right)^{1/2}$, 
we deduce that as soon as
\begin{equation}
\label{eq4.6}
2 \delta_1 \leq \sqrt{\frac{\bigl[ \beta^{-1} 
+ 2 \log(\epsilon_2^{-1}) \bigr] v_0}{n}} 
\; \varphi \Biggl( 
\frac{(1 + \beta) \bigl[ \beta^{-1} + 2 \log(\epsilon_2^{-1})\bigr]}{2n} 
\Biggr)^{-1},
\end{equation}
which is the case at least as soon as
$$
\epsilon_2 \geq \exp \left( \frac{1}{2 \beta} - \frac{\ds n v_0}{ \ds
2 (1+\beta)^2 \delta_1^2 + 8(1+\beta) v_0} \right),
$$
with probability at least $1 - 2 \epsilon_1 - 2 \epsilon_2$, 
$$
\lvert m - \wh{\theta}_{\alpha}
\rvert \leq 
\frac{2}{(1+\beta)} \left( \frac{nv_0}{\beta^{-1} - 2 \log(\epsilon_2)} 
\right)^{1/2} \varphi \left( 
\frac{(1 + \beta) \bigl[ \beta^{-1} - 2 \log(\epsilon_2) \bigr]}{2n} \right).
$$
\end{prop}
\ 
\vspace{-6ex}

In the following plot, we took the same parameters as in the previous
section, and substituted only the last step. It shows some improvement,
especially for moderate confidence levels, but requires more involved 
computations.\\
\mbox{} \hfill \includegraphics{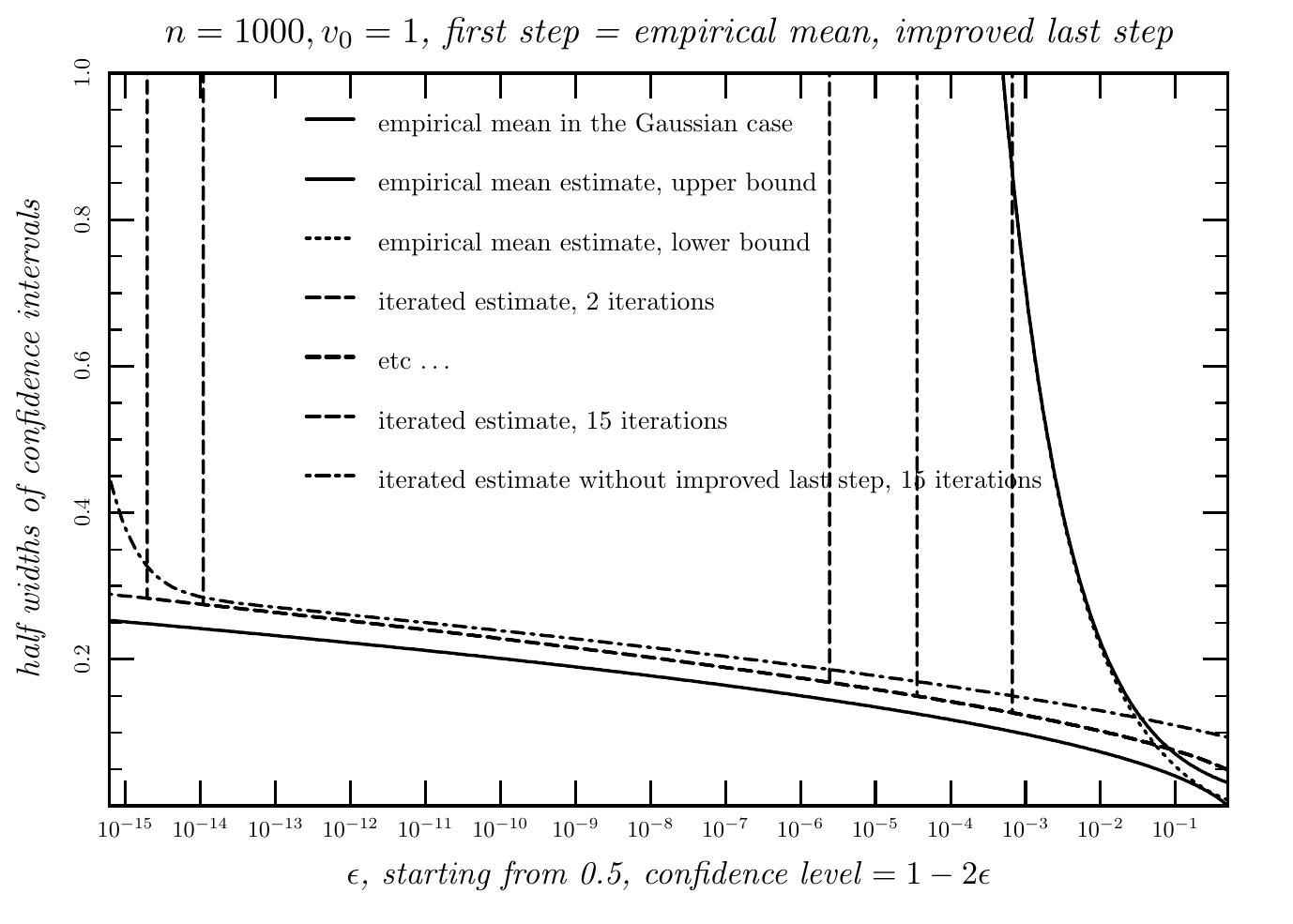}   
\hfill \mbox{}\\[-5ex]
(The vertical lines correspond to the confidence level after which condition 
\myeq{eq4.6} breaks.) This is what we obtain when we decrease $n$ to 
$300$,\\
\mbox{} \hfill \raisebox{-8ex}[0.85\height][0ex]{\includegraphics{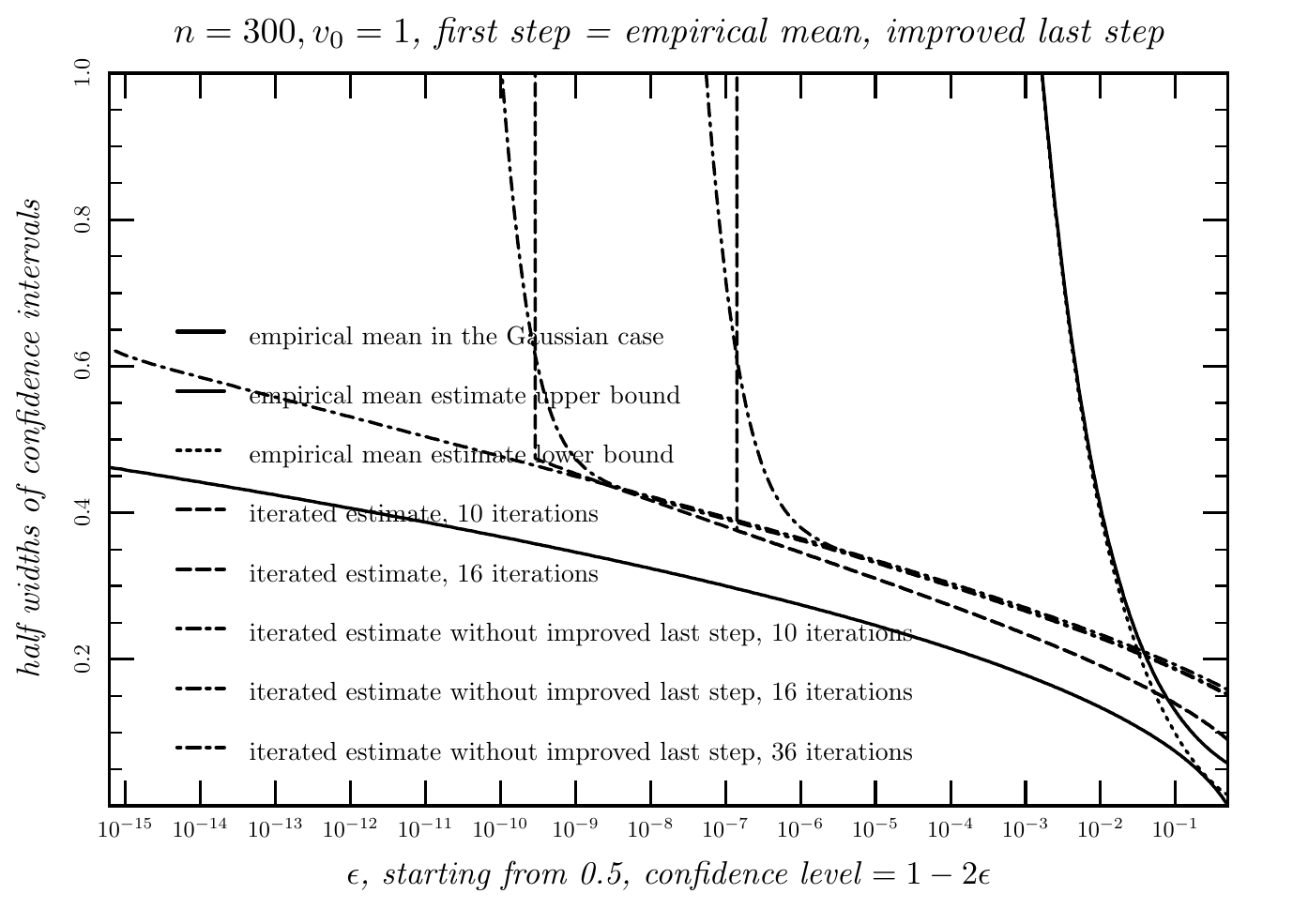}}   
\hfill \mbox{}\\
When the sample size is thus decreased, the last step improvement 
works up to $\epsilon \simeq 10^{-8}$, after which the iterated estimate
without last step improvement takes the lead, as shown on the previous plot. 

\section{Mean estimate from a kurtosis prior bound}

Situations where the variance is unknown are likely to happen.
It is possible to deal with them while making assumptions 
on the kurtosis. 

More precisely, let us introduce some {\em uniform kurtosis} coefficient,
that we define as 
$$
c = \sup_{\theta \in \B{R}} \frac{ \B{E} \bigl[ (Y-\theta)^4 \bigr]}{
\B{E} \bigl[ (Y - \theta)^2 \bigr]^2}.
$$
Its relation to the classical centered kurtosis 
$\ds \kappa = \frac{\B{E}\bigl[ (Y - m)^4 \bigr]}{\B{E} 
\bigl[ (Y - m)^2 \bigr]^2}$ is given by the following 
lemma.
\begin{lemma}
\label{lemma5.1}
The two kurtosis coefficients defined above satisfy  
the inequalities
$$
\kappa \leq c \leq 
\frac{1}{9}\biggl( \kappa^{1/2} + 2 \bigl( \kappa + 3 \bigr)^{1/2}
\biggr)^2 \leq \kappa + 2.
$$

On the other hand, if $\kappa_{\B{P}}$ and $c_{\B{P}}$ 
are the kurtosis and uniform kurtosis of the probability 
measure $\B{P}$ 
$$
\sup_{\B{P}} 
c_{\B{P}} - \kappa_{\B{P}} = 2,
$$
where the supremum is taken over all probability measures on the 
real line,
proving that the previous bound is tight in the worst case.

Anyhow, in the favourable case when 
the skewness is null, 
meaning that \linebreak  $\B{E} \bigl[ (Y-m)^3 \bigr] = 0$,
the two coefficients are equal whenever $\kappa \geq 3$, and 
more precisely
$$
c = 
\begin{cases}
\ds \kappa + \frac{(3-\kappa)^2}{5 - \kappa}, & 
\ds 1 \leq \kappa \leq 3,\\
\kappa, & \kappa \geq 3.
\end{cases}
$$ 
\end{lemma} 

Let us consider for any $\theta_0 \in \B{R}$
and $\delta \in )0,1)$ the estimator of the mean 
$\wh{\theta}_{\alpha}(\theta_0)$ 
already considered in previous sections 
and defined by equation \myeq{eq1.1.2}.
Let us also consider the increasing mapping $\alpha \in \B{R}_+ 
\mapsto Q_{\theta, \delta}(\alpha)$ defined as
$$
Q_{\theta,\delta}(\alpha) = \frac{1}{n} \sum_{i=1}^n \log \biggl\{ 
1 + \alpha(Y_i - \theta)^2 - \delta + \frac{1}{2} \Bigl[ 
\alpha(Y_i - \theta)^2 - \delta \Bigr]^2 \biggr\}.
$$
We will use the ancillary function 
$\ds
h(a,y) = \frac{\ds 4y}{\ds (1+y) \Biggl\{ 1 + 
\sqrt{1 - \frac{4ay^2}{(1+y)^2}}\Biggr\}}$.

The next proposition is concerned with random confidence
intervals, whose lengths are defined with the help of some 
estimator of the variance. 
More precisely, we are going to iterate a process where we successively 
estimate $v + (m-\theta)^2$ and $m$. 

\begin{prop}
\label{prop5.2}
Let us choose positive real constants $x_i$,  
and confidence levels $\epsilon_i$, $i = 1, \dots, 
2 k$. Let $U_i$, $i = 2, \dots, 2 k$ be uniform real random 
variables in the interval $(-1, +1)$, independent of each 
other and of everything else. Let us start with some prior 
guess $\theta_1$ for $m$ and let us define by induction the
sequence of values  
\begin{align*}
\wt{\theta}_1 & = \theta_1,\\
\delta_1 & = \sqrt{\frac{2 \log(\epsilon_1^{-1})}{(c-1) n}},\\
\zeta_1 & = - \frac{1}{2} \log \Bigl\{ 1 - h \Bigl[ \tfrac{c}{c-1}, 
(c-1) \delta_1 \Bigr] \Bigr\},\\
\wt{q}_1 & = \frac{\delta_1 \exp( - \zeta_1)}{Q_{\theta_1,\delta_1}^{-1}
\bigl[ -(c-1) \delta_1^2 \bigr]},\\
\wt{q}_2 & = \wt{q}_1 \exp( x_2 \zeta_1 U_2),\\
\gamma_2 & = \log(1 + x_2^{-1}),\\
\alpha_2 & = \exp \biggl[ - \frac{(1+x_2)\zeta_1}{2} \biggr] \sqrt{\frac{2 \bigl[ \log(\epsilon_2^{-1}) + \gamma_2 \bigr]}{
n \wt{q}_2}},\\
\zeta_2 & = \exp \biggl[ \frac{(1+x_2) \zeta_1}{2} \biggr] 
\sqrt{\frac{2 \wt{q}_2 \bigl[ \log(\epsilon_2^{-1}) + \gamma_2 \bigr]}{n}}, \\
\wt{\theta}_2 & = \wh{\theta}_{\alpha_2}(\theta_1),\\
& \vdots \\
\gamma_{2i - 1} & = \gamma_{2i - 2} + \log (1 + x_{2i-1}^{-1}),\\
\delta_{2i-1} & = \sqrt{\frac{2 \bigl[ \log(\epsilon_{2i-1}^{-1}) + 
\gamma_{2i-1} \bigr]}{(c-1) n}},\\
\wt{\theta}_{2i-1} & = \wt{\theta}_{2i - 2} + \zeta_{2i-2} x_{2i-1} U_{2i-1},\\ 
\zeta_{2i-1} & = - \frac{1}{2} \log \Bigl\{ 1 - h 
\Bigl[ \frac{c}{c-1}, (c-1) \delta_{2i-1} \Bigr] \Bigr\},\\
\wt{q}_{2i-1} & = \frac{\delta_{2i-1} \exp( - \zeta_{2i-1})}{
Q_{\wt{\theta}_{2i-1}, \delta_{2i-1}}^{-1} \bigl[ 
- (c-1) \delta_{2i-1}^2 \bigr]},\\
\wt{q}_{2i} & = \wt{q}_{2i-1} \exp( x_{2i} \zeta_{2i-1} U_{2i}),\\
\gamma_{2i} & = \gamma_{2i-1} + \log(1 + x_{2i}^{-1}),\\
\alpha_{2i} & = \exp \biggl[ -
\frac{(1 + x_{2i}) \zeta_{2i-1}}{2} \biggr] \sqrt{\frac{2 \bigl[ \log(\epsilon_{2i}^{-1}) + \gamma_{2i} 
\bigr]}{n \wt{q}_{2i}}},\\
\zeta_{2i} & = \exp \biggl[ \frac{(1 + x_{2i}) \zeta_{2i-1}}{2} 
\biggr] 
\sqrt{ \frac{2 \wt{q}_{2i} \bigl[ \log(\epsilon_{2i}^{-1}) 
+ \gamma_{2i} \bigr]}{n}}, \\
\wt{\theta}_{2i} & = \wh{\theta}_{\alpha_{2i}}(\wt{\theta}_{2i-1}),\\
& \vdots
\end{align*}
Let us remark that $\ds \gamma_i = \sum_{j=2}^i \log( 1 + x_j^{-1})$,
$\delta_{2i-1}$, and $\zeta_{2i-1}$ $i = 1, \dots, k$ are non random 
and known prior to observing the sample.\\
Let us assume that
$\ds
\max_{i=1, \dots, k} \delta_{2i-1} \leq \frac{1}{2\sqrt{c(c-1)} 
- (c-1)}.
$\\[1ex]
With probability at least $1 - 2 \sum_{i=1}^{2k} \epsilon_i$, 
for any $i=1, \dots, k$, 
\begin{gather*}
\lvert m - \wt{\theta}_{2i} \rvert \leq 
\zeta_{2i},\\ 
\bigl\lvert \log(\wt{q}_{2i-1}) - \log \bigl[ v 
+ (m - \wt{\theta}_{2i-1})^2 \bigr]  \bigr\rvert 
\leq \zeta_{2i-1}.
\end{gather*}
As a consequence, on the same event of probability 
at least $1 - 2 \sum_{i=1}^{2k} \epsilon_i$,  for 
any $i=2, \dots, k$, 
\begin{gather*} 
\bigl(m - \wt{\theta}_{2i-1} \bigr)^2 \leq 
(1  + x_{2i-1})^2 \zeta_{2i-2}^2,\\ 
\wt{q}_{2i-1} \leq \bigl[ v + (1 + x_{2i-1})^2 \zeta_{2i-2}^2 
\bigr] \exp( \zeta_{2i-1}),\\
\zeta_{2i} \leq \exp \bigl[ (1 + x_{2i}) \zeta_{2i-1} \bigr] 
\sqrt{\frac{2 \bigl[ v + (1+x_{2i-1})^2 \zeta_{2i-2}^2 \bigr] 
\bigl[ \log(\epsilon_{2i}^{-1}) + \gamma_{2i} \bigr]}{n}}, 
\\
\zeta_{2} \leq \exp \bigl[ (1 + x_2) \zeta_1 \bigr] \sqrt{ 
\frac{2 \bigl[ v + (m - \theta_1)^2 \bigr] \bigl[ \log(\epsilon_{2}^{-1}) 
+ \gamma_2 \bigr]}{n}}, \\
\exp( - \zeta_{2i-1}) \wt{q}_{2i-1} - (1 + x_{2i-1})^2 \zeta_{2i-2}^2 
\leq v \leq \exp(\zeta_{2i-1}) \wt{q}_{2i-1}.
\end{gather*}
These equations allow to compute by induction a 
(non observable) deterministic bound for $\zeta_{2k}$, 
which is itself a random observable confidence interval half width 
for the estimate of the mean given by $\wt{\theta}_{2k}$. 
The last equation (better used with $i = k$) shows that 
we get as a by-product an estimate
of the variance with observable
as well as theoretical confidence bounds.
\end{prop}

In the sequel, we will give lower and upper bounds
for the worst case behaviour of the empirical mean depending on the variance
and the kurtosis. Note that here we do better, since we 
also estimate the variance and assume only a known prior bound 
on the kurtosis. Obtaining a similar observable confidence interval for 
the empirical mean would require to estimate the variance 
under a kurtosis bound, which is not something straightforward,
as will also be discussed a little later. In the following 
plot, we chose a sample of size $2000$, a size where things
start to behave nicely under these assumptions. 

Although the lower and upper deviation bounds shown for the empirical 
mean estimator do not correspond to observable confidence 
intervals, we see that for confidence levels higher 
than $1 - 2\cdot10^{-8}$, 
the observable confidence interval of our estimator 
outperforms the deviations of the empirical mean, 
up to confidence levels as high as $1 - 2 \cdot 10^{-14}$.
We also plotted the upper estimate for the
standard deviation (assumed to be equal to one).
We took $x_i = 0.5$, $i < 2k$, $x_{2k} = 0.1$ and assumed 
that $\lvert m - \theta_1 \rvert \leq 10 \sqrt{v}$.  
We kept the kurtosis to $3$, the kurtosis of the Gaussian 
distribution.
\\
\mbox{} \hfill \includegraphics{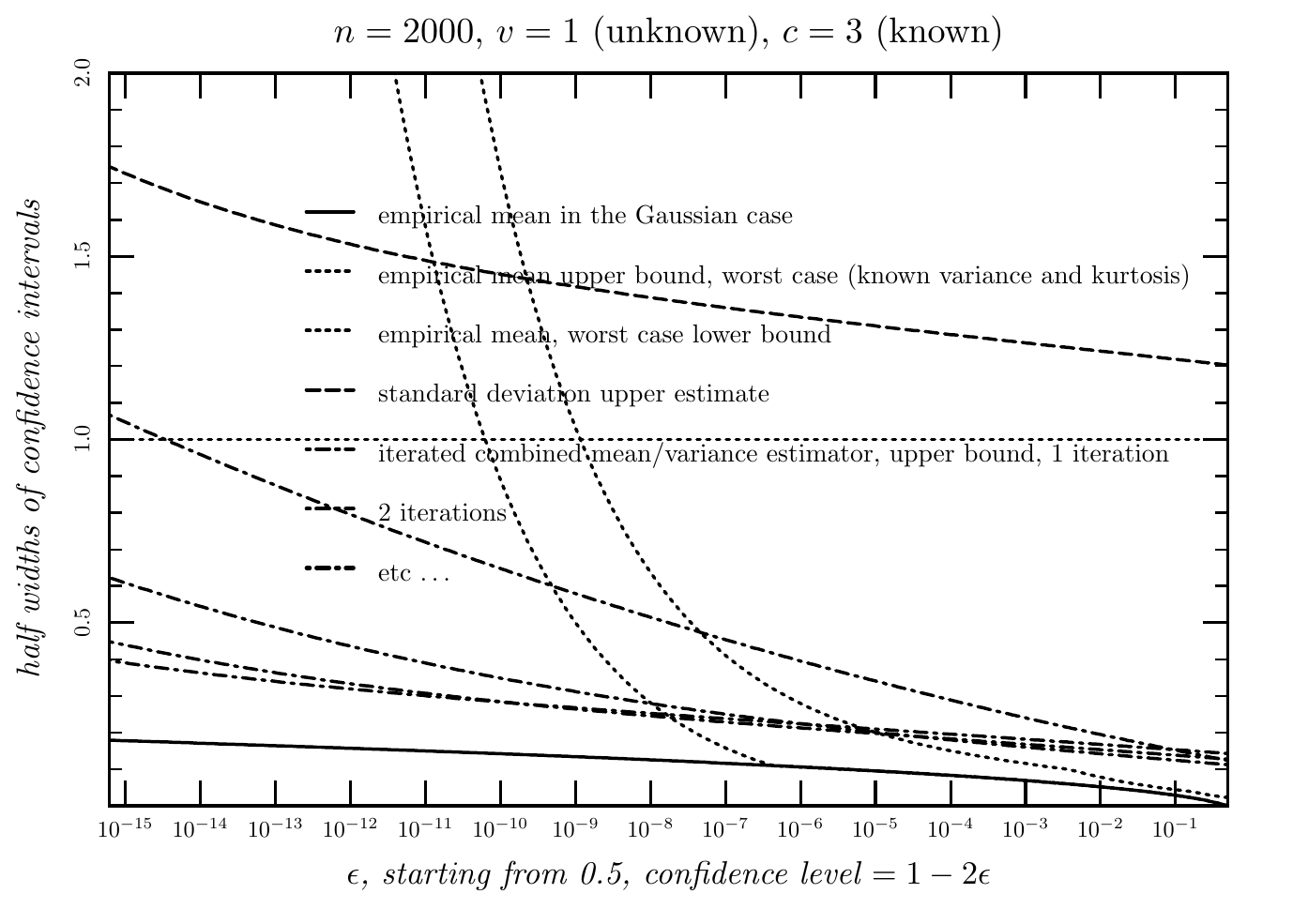}  
\hfill \mbox{}\\[-4ex]
This is now what happens when we increase $c$ to $6$ (
taking this time $x_{i} \equiv 0.1$).\\[-1ex]
\mbox{} \hfill \raisebox{-2ex}[\height][0pt]{\includegraphics{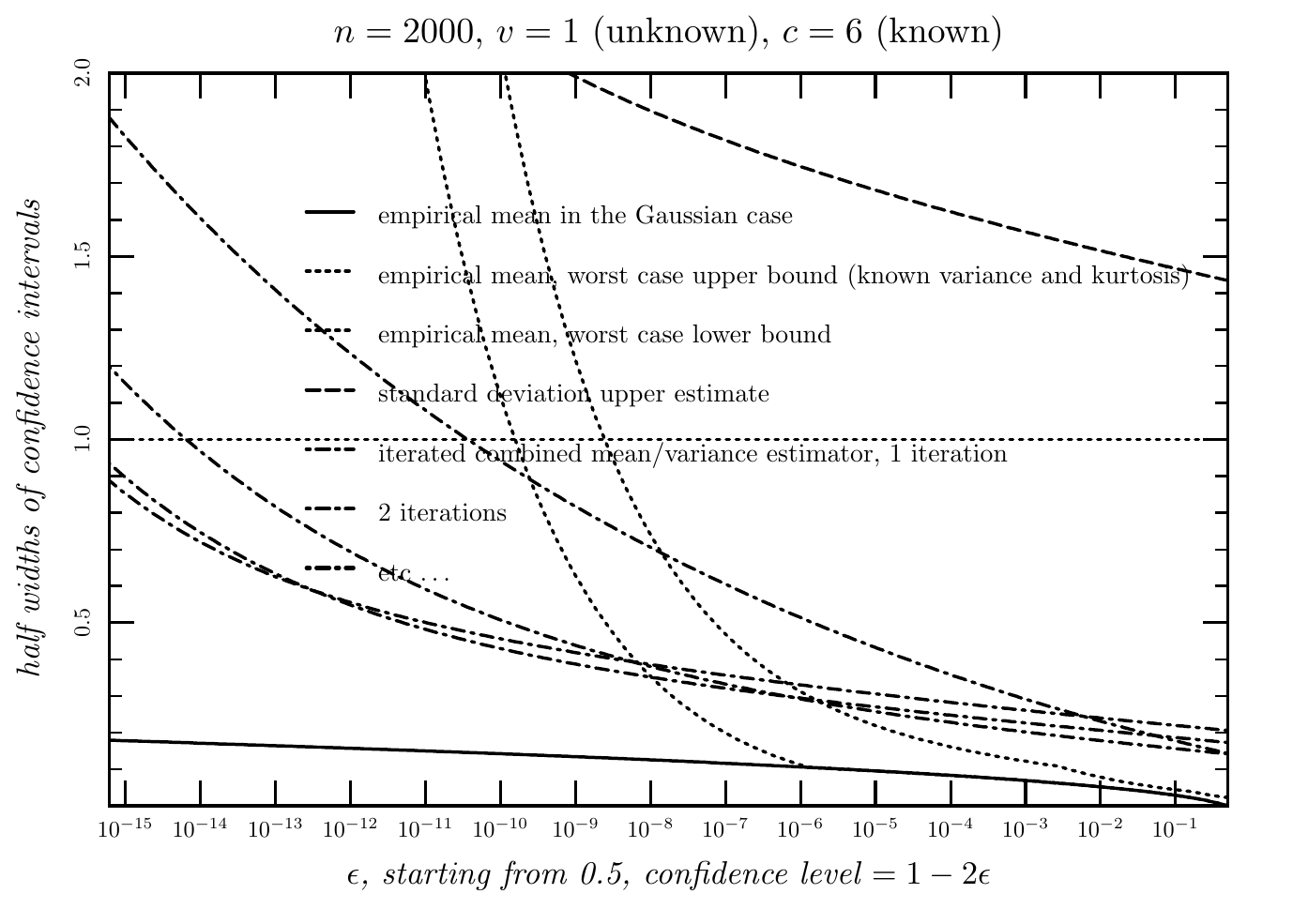}}  
\hfill \mbox{}\\
When we increase the sample size to $n = 5000$, this makes things easier:\\[-2ex]
\mbox{} \hfill \raisebox{-2ex}[\height][0pt]{\includegraphics{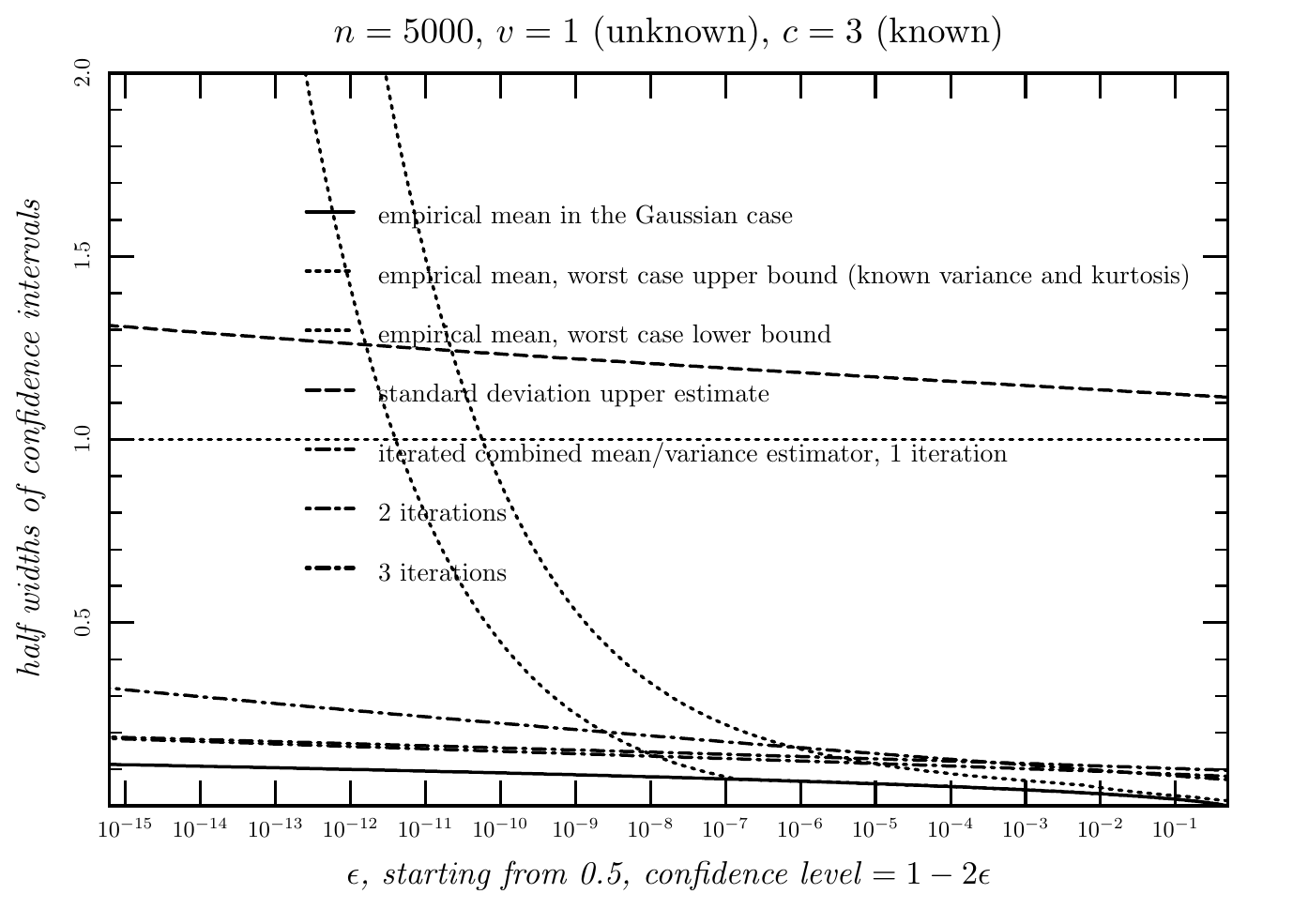}}  
\hfill \mbox{}\\[-2ex]
This is the influence of the kurtosis on the bounds for a sample 
of size $n = 5000$:\\
\mbox{} \hfill \raisebox{-7ex}[0.9\height][0pt]{\includegraphics{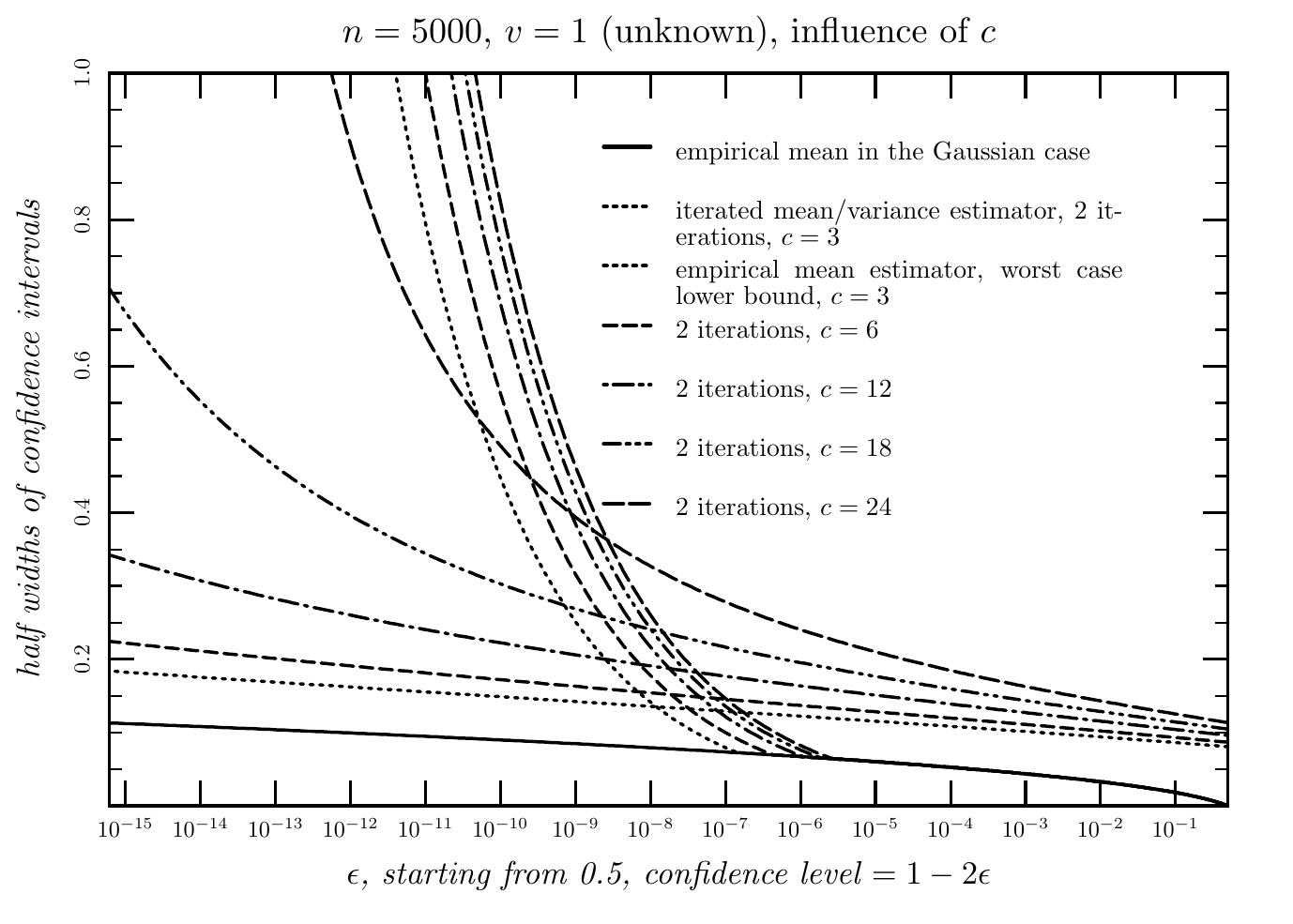}}  
\hfill \mbox{}
\section[Adapting to an unknown variance]{Adapting to an unknown variance when the kurtosis is 
unknown or even infinite}

In this section, we will point out that Lepski's 
renowned adaptation method \cite{Lepski2} can be put to good use
when nothing is known, neither the variance (still 
assumed to be finite) nor the kurtosis (not even assumed 
to be finite !). Of course, under so uncertain, (but unfortunately
so frequent) circumstances, it is not possible to
provide any {\em observable} confidence level. Nevertheless, 
it is still possible to adapt to the variance and to give  
deviation bounds depending on the unknown variance.
Here, a clear distinction should be made between {\em adapting} 
to the variance and {\em estimating} the variance : 
estimating the variance at any predictable rate is impossible 
in this context where we do not assume any higher moment to be 
bounded.

The idea of Lepski's method is powerful and simple : consider
a sequence of confidence intervals obtained by assuming a variance
bound $v_0$ to take a range of possible values and pick up as an estimator
the middle of the smallest interval intersecting all the larger 
ones. For this to be legitimate, we need all the confidence regions
for which the variance bound is valid to hold together, which 
is performed using a union bound.

Let us describe this idea more precisely. Let $\wh{\theta}(v_0)$ 
be some estimator of the mean depending on some assumed variance bound 
$v_0$, as the ones described in the beginning of this paper. 
Let $\delta(v_0, \epsilon) \in \B{R}_+ \cup \{ + \infty \}$  
be the corresponding confidence bound : 
namely let us assume that with probability at least $1 - 2 \epsilon$, 
$$
\lvert m -  \wh{\theta}(v_0) \rvert \leq \delta(v_0, \epsilon).
$$
Presumably, except for distributions with bounded support, 
$\delta(v_0, 0) = + \infty$.

Let $\nu \in \C{M}_+^1(\B{R}_+)$ be some coding atomic sub-probability 
measure on the positive real line, which will serve to take 
a union bound on a (countable) set of possible values of $v_0$.

We can choose for instance for $\nu$ the following coding distribution : 
expressing $v_0$ by comparison with some reference value 
$V$ 
$$
v_0 = V 2^s \sum_{k=0}^{d} c_k 2^{-k}, \quad
s \in \B{Z}, d \in \B{N}, (c_k)_{k=0}^d \in \{0,1\}^{d+1},
c_0 = c_d = 1,
$$
we set $\nu(v_0) = 
\bigl[ (\lvert s \rvert + 2) (\rvert s \rvert + 3)
(d+1)(d+2) 2^{d-1} \bigr]^{-1}$, and otherwise we set $\nu(v_0) = 0$. 
It is easy to see that this defines a subprobability distribution on
$\B{R}_+$ (supported by dyadic numbers scaled by the factor V). 
It is clear that, as far as possible, the reference value $V$ should
be chosen as close as possible to the true variance $v$.  

Let us consider for any $v_0$ such that $\delta(v_0, \epsilon \nu(v_0))
< + \infty$ the confidence interval 
$$
I(v_0) = \wh{\theta}(v_0) + \delta \bigl[ v_0, \epsilon \nu(v_0) 
\bigr] \times (-1,1).
$$
Let us put $I(v_0) = \B{R}$ when $\delta(v_0, \epsilon \nu(v_0)) = + 
\infty$. 

Let us consider the non-decreasing family of closed intervals
$$
J(v_1) = \bigcap \Bigl\{ I(v_0) : v_0 \geq v_1 \Bigr\}, \qquad v_1 \in \B{R}_+.
$$
A union bound shows immediately that
with probability at least $1 - 2 \epsilon$, $m \in J(v)$, 
implying as a consequence that $J(v) \neq \varnothing$.

\begin{prop}
Since $v_1 \mapsto J(v_1)$ is a non decreasing family of closed intervals, 
the intersection 
$$
\bigcap \Bigl\{ J(v_1) : v_1 \in \B{R}_+, J(v_1) \neq \varnothing \Bigr\}
$$
is a non empty closed interval, and we can therefore pick up an adaptive 
estimator $\wt{\theta}$ belonging to it, choosing for instance 
the middle of this interval.

With probability at least $1 - 2 \epsilon$, $ m \in J(v)$, which implies
that $J(v) \neq \varnothing$, and therefore that $\wt{\theta} 
\in J(v)$. 

Thus with probability at least $1 - 2 \epsilon$
$$
\lvert m - \wt{\theta} \rvert \leq \lvert J(v) \rvert 
\leq 2 \inf_{v_0 > v}  \delta(v_0, \epsilon \nu(v_0)). 
$$ 
If the confidence bound $\delta(v_0, \epsilon)$ is 
homogeneous, in the sense that 
$$
\delta(v_0, \epsilon) = \delta(1,\epsilon) \sqrt{v_0}, 
$$
as it is the case in Proposition \thmref{prop3.1.2} and 
Proposition \thmref{prop1.4} when used in conjunction 
with Proposition \ref{prop3.1.2}, then with probability at least 
$1 - 2 \epsilon$, 
$$
\lvert m - \wt{\theta} \rvert \leq 2 \inf_{v_0 > v}  
\delta(1,\epsilon \nu(v_0)) \sqrt{v_0}.
$$
\end{prop}
Since usually $\epsilon \mapsto \delta(1, \epsilon)$ is quite 
flat in the high confidence region, as shown on previous plots, 
we see that, in the high confidence region we are
mostly interested in in this paper, 
the order of magnitude of the adaptive confidence 
bound is not much more than twice the value $\delta(v,\epsilon)$
of the confidence bound we would have obtained for the estimator
$\wh{\theta}(v)$ which we could have used had we known the 
exact value of the variance beforehand.

\section{Worst case empirical mean deviations for a given kurtosis value}

In the previous sections, we studied truncation techniques
suited to various prior hypotheses on the sample distribution.
It is interesting to compare them to the performance of the 
empirical mean estimator. This section is devoted to 
upper bounds, whereas the next will study corresponding lower 
bounds.

When the variance is known and nothing else, it is easy to see,
using Chebyshev's inequality for the second moment
that the empirical mean
$$
M = \frac{1}{n} \sum_{i=1}^n Y_i
$$  
is such that 
\begin{equation}
\label{eq7.1}
\B{P} \biggl( \lvert M - m \rvert \geq \sqrt{\frac{v}{2 \epsilon n}} 
\biggr) \leq 2 \epsilon.
\end{equation}

The behaviour of the empirical mean for a given kurtosis is
not so straightforward. The following bound uses a truncation 
argument, allowing to study separately the behaviour of large
and rare values. It is to our knowledge a new result. 
We will show later in this paper that its leading term is essentially 
tight (up to the $(3/2)^{1/4}$ multiplicative constant due to the 
union bound argument). 

\begin{prop}
\label{prop6.1}
For any probability distribution whose kurtosis is not greater than
$\kappa$, the empirical mean $M$ is such that with probability at 
least $1 - 2 \epsilon$, 
\begin{multline*}
\frac{\lvert M - m \rvert}{\sqrt{v}} \leq \frac{2 \log(\frac{3}{2} 
\epsilon^{-1}) \sqrt{\kappa}}{5n} +
\sqrt{\frac{2 \log(\frac{3}{2} \epsilon^{-1})}{n}}  
\\ + \Biggl( \frac{3 \kappa}{2 \epsilon n^3} \Biggr)^{1/4} 
\Biggl( 1 + \frac{3^5(n-1) \log(\frac{3}{2}\epsilon^{-1})^2 
\kappa}{2500 n^2} + \frac{12 \sqrt{2} \log ( \frac{3}{2} 
\epsilon^{-1})^{3/2} \sqrt{\kappa}}{25 n^{3/2}} \Biggr)^{1/4}.
\end{multline*}
\end{prop}

Let us also stress here the fact that estimating the variance 
under a kurtosis bound, using the empirical estimator 
$$
M_2 = \frac{1}{n} \sum_{i=1}^n Y_i^2
$$
of the moment of order two is likely to be unsuccessful at 
high confidence levels. Indeed, computing the quadratic
mean 
$$
\B{E} \Bigl\{ \bigl[ M_2 - \B{E}(Y^2) \bigr]^2 \Bigr\} 
= \frac{\B{E}(Y^4) - \B{E}(Y^2)^2}{n} \leq \frac{(c-1)}{n} \B{E}(Y^2)^2,
$$
we can only conclude, using Chebyshev's inequality,  
that with probability at least $1 - 2 \epsilon$
$$
\B{E}(Y^2) \leq \frac{M_2}{1 - \sqrt{\frac{c-1}{2n\epsilon}}},
$$
a bound which breaks down at level of confidence 
$\ds \epsilon = \frac{c-1}{2n}$, and which we do not suspect to 
be substantially improvable in the worst case. 
In contrast to this, Proposition \thmref{prop5.2} 
provides a variance estimator at high confidence levels. 

\section{Lower bounds}

\subsection{Lower bound for Gaussian distributions}

This lower bound is well known. We recall it here for the sake
of completeness. 

The empirical mean cannot
be improved in the Gaussian case in the following 
precise sense.

\begin{prop}
\label{prop2.1.2}
For any estimator of the mean $\wh{\theta} : \B{R}^n \rightarrow \B{R}$, 
any variance value $v > 0$, 
and any deviation level $\eta > 0$, 
there is some Gaussian measure $\C{N}(m,v)$ (with variance $v$ and mean $m$)
such that the i.i.d. sample of length $n$ drawn from this distribution 
is such that 
$$
\B{P} \bigl( \wh{\theta} \geq m + \eta \bigr) \geq \B{P} 
\bigl( M \geq m + \eta \bigr) 
\quad \text{ or } \quad 
\B{P} \bigl( \wh{\theta} \leq m - \eta \bigr) \geq 
\B{P} \bigl( M \leq m - \eta \bigr),
$$
where $\ds M = \frac{1}{n} \sum_{i=1}^n Y_i$ is the empirical mean.
\end{prop}

This means that any distribution free symmetric confidence interval based on 
the (supposedly known) value of the variance has to include the 
confidence interval for the empirical mean of a Gaussian distribution, 
whose length is exactly known and equal to the properly
scaled quantile of the Gaussian measure. 

Let us state this more precisely. With the 
notations of the previous proposition 
\begin{multline*}
\B{P}(M \geq m + \eta) = \B{P} \bigl( M \leq m - \eta) 
\\ = G\left[ \left( \sqrt{\frac{n}{v}} \eta , + \infty \right( 
\right] = 1 - F \left( \sqrt{\frac{n}{v}} \eta \right),
\end{multline*}
where $G$ is the standard normal measure and $F$ its distribution 
function. 

The upper bounds proved in this paper can be decomposed into 
$$
\B{P}(\wh{\theta} \geq m + \eta) \leq \epsilon \qquad \text{and} 
\qquad \B{P} (\wh{\theta} \leq m - \eta) \leq \epsilon,
$$
although we preferred for simplicity to state them 
in the slightly weaker form $\B{P}( \lvert 
\theta - m \rvert \geq \eta) \leq 2 \epsilon$.

As the Gaussian shift model made of Gaussian distributions with 
a given variance and varying means, is included in all the 
models we consider to state bounds, we necessarily should have 
according to the previous proposition
$$
\epsilon \geq 1 - F \left( \sqrt{\frac{n}{v}} \eta \right), 
$$
which can be also written as 
$$
\eta \geq \sqrt{\frac{v}{n}} F^{-1} ( 1 - \epsilon ).
$$

Therefore some visualisation of the quality of our bounds can be
obtained by plotting $\ds \epsilon \mapsto \eta$ against 
$\ds \epsilon \mapsto \sqrt{\frac{v}{n}} F^{-1}(1 - \epsilon)$, 
as we did in the previous sections. 

\subsection{Worst performance of the empirical mean for a given 
variance}
Another way to measure the quality of the bound is to compare
it to the empirical mean outside from the Gaussian shift 
model, where we have seen that the deviations of the empirical mean 
are minimax at any confidence level. This is done in the following proposition.
\begin{prop}
\label{prop2.2}
For any value of the variance $v$, any deviation level $\eta > 0$,
there is some distribution with variance $v$ and mean $0$ such that
the i.i.d. sample of size $n$ drawn from it satisfies  
$$
\B{P}(M \geq \eta) = \B{P}(M \leq - \eta) \geq 
\frac{v \left( 1 - \frac{v}{\eta^2 n^2} \right)^{n-1}}{2 n \eta^2}.
$$
Thus, as soon as $\epsilon \leq (2e)^{-1}$, 
with probability at least $2 \epsilon$, 
$$
\lvert M - m \rvert \geq  \sqrt{\frac{v}{2 n \epsilon}} 
\left( 1 - \frac{2e \epsilon}{n} \right)^{\frac{n-1}{2}}.
$$
\end{prop}
Let us remark that this bound is pretty tight, as shown in the 
next plot, since,
according to equation \myeq{eq7.1}
with probability at least $1 - 2 \epsilon$, 
$$
\lvert M - m \rvert \leq \sqrt{\frac{v}{2n\epsilon}}.
$$
\mbox{} \hfill \raisebox{-4ex}[0.97\height][0ex]{
\includegraphics{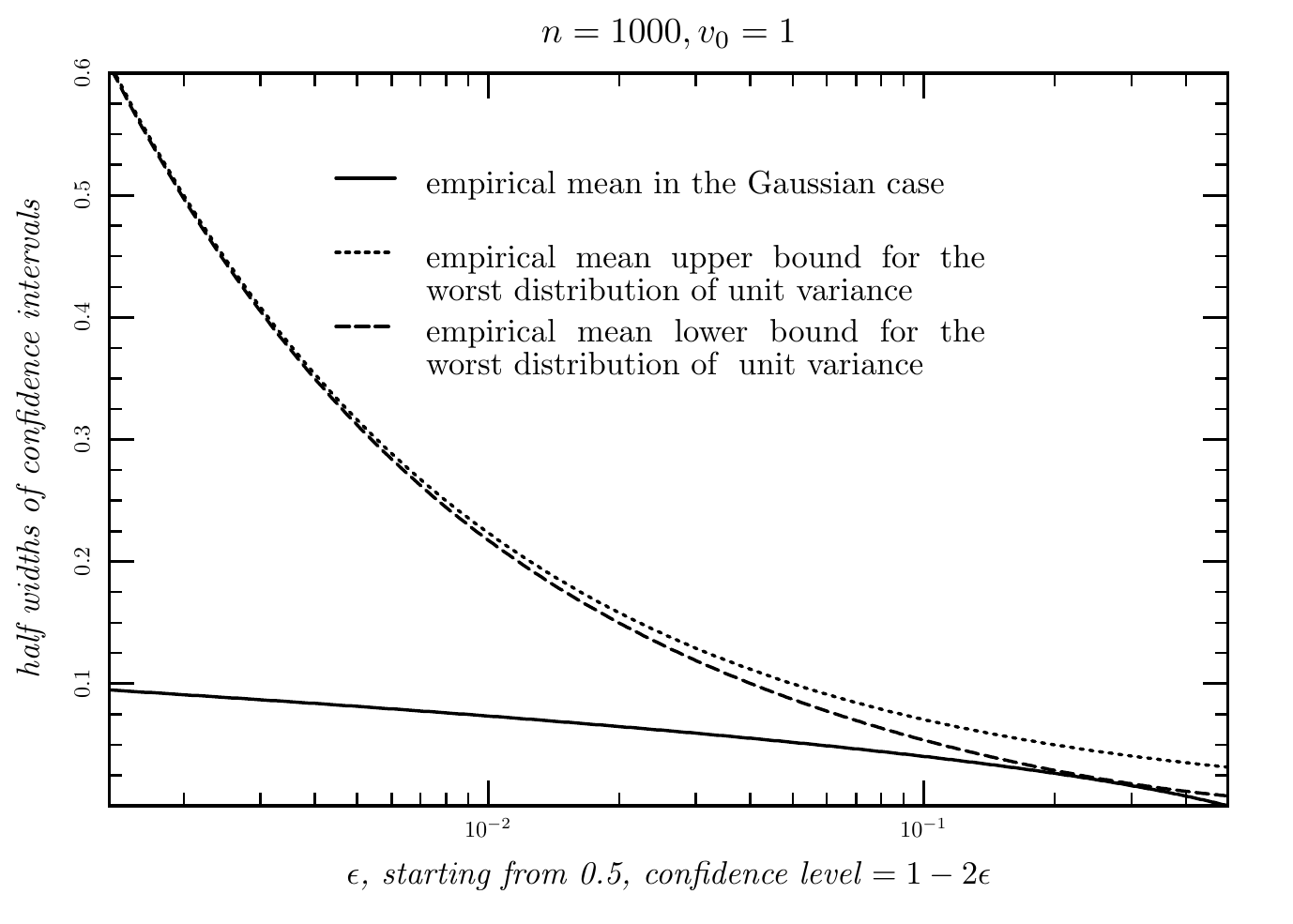}}
\hfill \mbox{}

\subsection{Worst performance of the empirical mean for a given kurtosis}
\begin{prop}
\label{prop7.3}
For any $c \geq 1 + 1/n$, and any $\ds \epsilon \leq (4e)^{-1}$, 
there is a probability measure on the real line, 
with uniform kurtosis equal to $c$  and unit variance, 
such that with probability at least $2 \epsilon$, 
$$
\lvert M - m \rvert \geq \biggl( \frac{c-1}{4 n^3 \epsilon} \biggr)^{1/4} 
\left( 1 -  \frac{4 e \epsilon}{n} \right)^{(n-1)/4}.
$$ 
\end{prop}

Let us plot this lower bound as well as the corresponding upper bound
given by Proposition \thmref{prop6.1}, for a sample of size $n = 2000$
and a kurtosis $c=6$. The space between the two curves is of moderate
size, showing that we got the order of magnitude right in these bounds.\\
\mbox{} \hfill \includegraphics{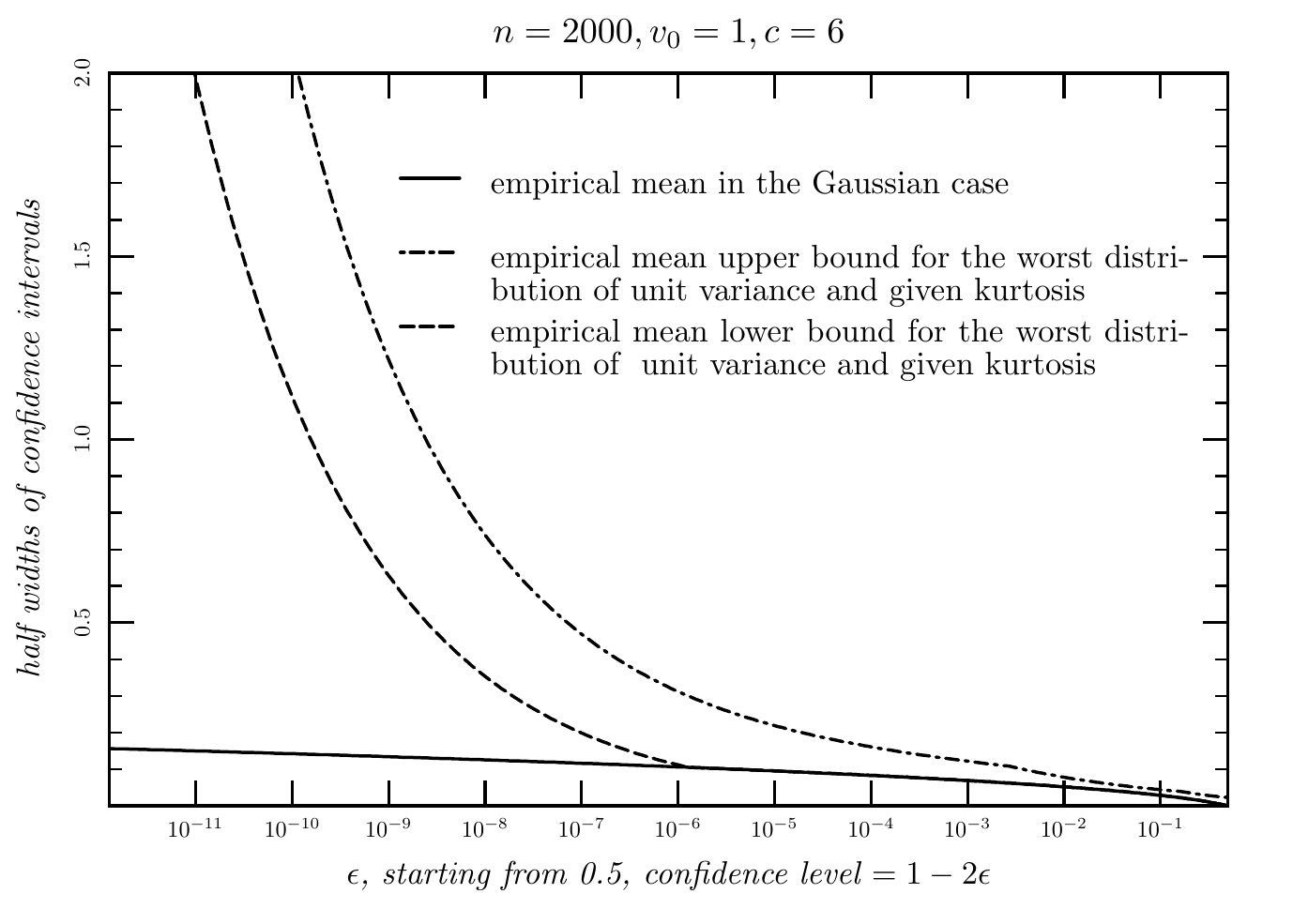}  
\hfill \mbox{}\\[-9ex]

\section{Proofs}

\subsection{Proof of Proposition \thmref{prop3.1}}

Let us start with some bounds for the map $x \mapsto \log \bigl( 1 + x + \frac{x^2}{2} 
\bigr)$. 
\begin{lemma}
\label{lemma9.1}
The map $x \mapsto \log \bigl( 1 + x + \frac{x^2}{2} \bigr)$ satisfies 
for any $x \in \B{R}$, 
$$
- \frac{x^4}{38} \leq \log \biggl( 1 + x + \frac{x^2}{2} 
\biggr) - x + \frac{x^3}{6} \leq \frac{x^4}{6}.
$$
\end{lemma}
\begin{proof}
Let us consider for some positive real parameter $a$ the function 
$$
f(x) = \log \bigl( 1 + x + \tfrac{x^2}{2} \bigr) - x + \frac{x^3}{6} 
-  \frac{a x^4}{8}.
$$
We can study its sign through its derivative
\begin{multline*}
f'(x) = \frac{1 + x}{1 + x + \frac{x^2}{2}} - 1 + \frac{x^2}{2} 
- \frac{a x^3}{2} =  \frac{x^2 \bigl( x + \frac{x^2}{2} \bigr) }{2 
\bigl( 1 + 
x + \frac{x^2}{2} \bigr)} - \frac{a x^3}{2} \\ = \frac{x^3 \bigl( 1 + \frac{x}{2} - a - ax - \frac{a x^2}{2} \bigr)}  {2 \bigl( 
1 + x + \frac{x^2}{2} \bigr)} 
= - \frac{x^3 \bigl[(a-1) + (a - \frac{1}{2}) x + \frac{a x^2}{2} \bigr]}{ 
2 \bigl( 1 + x + \frac{x^2}{2} \bigr)}.
\end{multline*}
When $(a-\frac{1}{2})^2 - 2a(a-1) \leq 0$, $f'(x)$ has the same sign as 
$- x$, showing that $\sup_{\B{R}} f = 0$, since $f(0) = 0$.
This condition can also be written as 
$ a^2 - a - \frac{1}{4} \geq 0$, and is fulfilled when 
$a = \frac{1 + \sqrt{2}}{2}$.  
Thus 
$$
\log \bigl( 1 + x + \tfrac{x^2}{2} \bigr) - x + \frac{x^3}{6} 
\leq \frac{(1 + \sqrt{2}) x^4}{16} \leq \frac{x^4}{6}, 
\qquad x \in \B{R}.
$$
Let us proceed to the lower bound now. 
Consider the same computations as above, but with a negative parameter $a$. 
In this case, under the same discriminant condition, $f'(x)$ has the 
same sign as $x$, showing that 
$\inf_{\B{R}} f = 0$. For the lower bound, we can thus take 
$a = \frac{1 - \sqrt{2}}{2}$, proving that
$$
- \frac{x^4}{38} \leq - \frac{\bigl( \sqrt{2}-1 \bigr) x^4}{16} \leq 
\log \bigl( 1 + x + \frac{x^2}{2} \bigr) - x + \frac{x^3}{6}, \qquad 
x \in \B{R}.
$$
\end{proof}

We will also need the following property of the {\em truncated 
exponential function} $\ds \frac{1}{2} 
\leq 1 +  x + \frac{x^2}{2} \simeq \exp(x)$:
\begin{equation}
\label{eq9.1.2}
- \log \biggl( 1 - x + \frac{x^2}{2} \biggr) 
= \log \biggl( \frac{1 + x + \frac{x^2}{2}}{1 + \frac{x^4}{4}} \biggr) 
\leq \log \biggl( 1 + x + \frac{x^2}{2} \biggr), \qquad x \in \B{R},
\end{equation}
so that
$$
- \log \biggl( 1 - x + \frac{x^2}{2} \biggr) 
\overset{\text{def}}{=} T_-(x) \leq T(x) \leq T_+(x) \overset{\text{def}}{=} \log \biggl( 1 + x + \frac{x^2}{2} \biggr).
$$
Accordingly
$$
\wh{\theta}_{\alpha}(\theta_0) \leq \theta_0 
+ \frac{1}{n \alpha} \sum_{i=1}^n T_+\bigl[ \alpha(Y_i - \theta_0) \bigr].
$$
We can then compute the exponential moment
\begin{multline*}
\B{E} \biggl\{ \exp \biggl[ \sum_{i=1}^n 
T_+\bigl[ \alpha(Y_i - \theta_0) \bigr] \biggr] \biggr\} 
\\ = \prod_{i=1}^n \B{E} \biggl( 1 + \alpha(Y_i - \theta_0) 
+ \frac{\alpha^2}{2} (Y_i - \theta_0)^2 \biggr) 
\\ = \exp \Biggl[ n \log \biggl( 1 + \alpha(m-\theta_0) 
+ \frac{\alpha^2}{2} \bigl[v + (m-\theta_0)^2 \bigr] \biggr) \Biggr].
\end{multline*}
From the exponential Chebyshev inequality $\ds 
\frac{\B{P} (X \geq \eta)}{
\B{E} \bigl[ \exp(X) \bigr]} \leq \exp( - \eta)$, considering 
$\epsilon = \exp(- \eta)$, we deduce that with 
probability at least $1 - \epsilon$, 
\begin{multline}
\label{eq9.1}
\wh{\theta}_{\alpha}(\theta_0) \leq \theta_0 \\ + \frac{1}{n \alpha} 
\sum_{i=1}^n \log \biggl( 1 + \alpha(m - \theta_0) + \frac{\alpha^2}{2} 
\bigl[ v + (m - \theta_0)^2 \bigr] \biggr) + \frac{\log(\epsilon^{-1})}{n 
\alpha}.
\end{multline}
Let us now remark that for any $x \in \B{R}$, any $y \in \B{R}_+$,
according to Lemma \thmref{lemma9.1},
\begin{multline*}
\log \biggl( 1 + x + \frac{x^2}{2} + y \biggr)
\\ = \log \biggl( 1 + x + \frac{x^2}{2} \biggr) 
+ \log \biggl( 1 + \frac{y}{1 + x + \frac{x^2}{2}} \biggr)  
\\ \leq x - \frac{x^3}{6} + \frac{x^4}{6} + 
\frac{y}{1 + \frac{x^4}{4}} \biggl( 1 - x + \frac{x^2}{2} \biggr)
\\ \leq x - \frac{x^3}{6} + \frac{x^4}{6} + y - xy + \frac{x^2y}{2}.
\end{multline*}
Thus
\begin{multline*}
\wh{\theta}_{\alpha}(\theta_0) \leq m + \frac{\alpha v}{2} + 
\frac{\log(\epsilon^{-1})}{n \alpha} \\ - \frac{\alpha^2 (m - \theta_0)^3}{6} 
+ \frac{\alpha^3 (m-\theta_0)^4}{6} - \frac{\alpha^2 v}{2}  
\biggl( m - \theta_0 - 
\frac{\alpha (m- \theta_0)^2}{2} \biggr).
\end{multline*}
In the same way, considering $\theta_0 - Y_i$ instead of $Y_i - \theta_0$
and using the symmetry of $T(x)$, we get with probability at least $1 - \epsilon$
\begin{multline*}
\wh{\theta}_{\alpha}(\theta_0)
\geq m - \frac{\alpha v}{2} - \frac{\log(\epsilon^{-1})}{n \alpha} 
\\ - \frac{\alpha^2 (m-\theta_0)^3}{6} - 
\frac{\alpha^3 (m-\theta_0)^4}{6} - \frac{\alpha^2 v}{2} 
\biggl( m - \theta_0 + \frac{\alpha (m - \theta_0)^2}{2} \biggr).
\end{multline*}
Therefore with probability at least $1 - 2 \epsilon$, 
\begin{multline*}
\lvert \wh{\theta}_{\alpha}(\theta_0) - m \rvert \leq 
\frac{\alpha v}{2} + \frac{\log(\epsilon^{-1})}{n \alpha} 
\\ + \frac{\alpha^2 \lvert m - \theta_0\rvert^3}{6}(1 
+ \alpha \lvert m - \theta_0 \rvert ) 
+ \frac{\alpha^2 \lvert m - \theta_0 \rvert v}{2} \biggl( 1  + \frac{\alpha  \lvert 
m - \theta_0  \rvert}{2} \biggr) 
\\ \leq \frac{\alpha v}{2} + \frac{\log(\epsilon^{-1})}{n \alpha} 
\\ + \frac{\alpha^2 \lvert m - \theta_0 \rvert}{2} 
\bigl( 1 + \alpha \lvert m - \theta_0 \rvert \bigr)\biggl( 
\frac{(m - \theta_0)^2}{3} + v \biggr) .
\end{multline*}
Specifically, when $\ds \alpha = \sqrt{\frac{2 \log(\epsilon^{-1})}{n v_0}}$,
with $v_0 \geq v$,  
\begin{multline*}
\lvert \wh{\theta}_{\alpha}(\theta_0) - m \rvert \leq 
\sqrt{ \frac{2 v_0 \log(\epsilon^{-1})}{n}} 
\\ + \frac{\log(\epsilon^{-1}) \lvert m - \theta_0\rvert}{3 n v_0} \bigl[ 
(m - \theta_0)^2 + 3 v_0 \bigr] \Biggl[ 1 + \lvert m - \theta_0 \rvert 
\sqrt{\frac{2 \log(\epsilon^{-1})}{
n v_0}} \Biggr].
\end{multline*}

If we prefer to keep the estimator $\wh{\theta}_{\alpha}(\theta_0)$ 
independent from the confidence level $1 - \epsilon$, 
we can choose $\ds \alpha = \sqrt{\frac{2}{n v_0}}$ and obtain 
for this value and with probability at least $1 - 2 \epsilon$, 
\begin{multline*}
\lvert \wh{\theta}_{\alpha}(\theta_0) - m \rvert 
\leq \bigl[ 1 + \log(\epsilon^{-1}) \bigr] 
\sqrt{\frac{v_0}{2 n}} 
\\ + \frac{\lvert m - \theta_0 \rvert}{3 n v_0} 
\bigl[ (m - \theta_0)^2 + 3 v_0 \bigr] \Biggl[ 
1 + \lvert m - \theta_0 \rvert \sqrt{\frac{2}{n v_0}} \Biggr].
\end{multline*}

\subsection{Proof of Proposition \thmref{prop3.2}}

If $\lvert m - \theta_0 \rvert$ is already small,
or if you are aiming at an iterative scheme, 
you can be content with the inequality
$$
\lvert \wh{\theta}_{\alpha}(\theta_0) - m \rvert \leq 
\frac{\alpha}{2}\bigl[ v_0 + (m- \theta_0)^2 \bigr] 
+ \frac{\log(\epsilon^{-1})}{n \alpha},
$$
which holds with probability at least $1 - 2 \epsilon$. 
This is a consequence of Equation \myeq{eq9.1} and the 
coarse inequality $\log(1 + x) \leq x$. 

\subsection{Proof of Proposition \thmref{prop2.1}}

We are going here to iterate the use of Proposition 
\thmref{prop3.2}. Applying it once, to start with, 
we get that with probability at least $1 - 2 \epsilon_1$
$$
\lvert \wt{\theta}_1 - m \rvert \leq \delta_1.
$$
Let
$$
\ov{\theta}_2 = \begin{cases}
\wt{\theta}_1 + \delta_1 x_2 U_2 & \text{when } \lvert \wt{\theta}_1 - m \rvert \leq 
\delta_1, \\ 
m + \delta_1 x_2 U_2, & \text{otherwise.}
\end{cases}
$$ 
We are going to use some PAC-Bayesian theorem to overcome
the fact that the sequence of estimators $\wt{\theta}_i$
is computed on the same sample.

Let us consider the prior distribution
$\pi_1$ defined as the uniform probability measure
on the interval $m + (1 + x_2) \delta_1 \times (-1, +1)$.
Let $\rho_1$ be the conditional distribution of $\ov{\theta}_2$
knowing the sample. 
From the definition of $\ov{\theta}_2$,
we see that for any value of the sample $(Y_i)_{i=1}^n$, 
the support of $\rho_1$ is included in the support
of $\pi_1$, and therefore that $\rho_1$ is absolutely continuous
with respect to $\pi_1$, with density
$$
\frac{d \rho_1}{d \pi_1} = 1 + x_2^{-1}, \qquad \rho_1 \text{ almost 
surely.} 
$$ 
Let us define the family of random variables
$$
X(\theta) = \sum_{i=1}^{n} T_+ \bigl[ \alpha_2 (Y_i - \theta) \bigr] 
- n \log \biggl( 1 + \alpha_2 (m - \theta) + \frac{\alpha_2^2}{2} 
\bigl[ v + (m - \theta)^2 \bigr] \biggr). 
$$
Integrating with respect to $\rho_1$, and using Fubini's theorem 
we get 
\begin{multline*}
\B{E} \Bigl\{ \tint \rho_1(d \theta) \exp \bigl[ X(\theta) 
- \log(1 + x_2^{-1}) \bigr] \Bigr\} \\ = 
\B{E} \Biggl\{ \int \rho_1(d \theta) \B{1} 
\biggl( \frac{d \rho_1}{d \pi_1}(\theta) > 0 \biggr) \exp \biggl\{ X(\theta) 
- \log \left[ \frac{d \rho_1}{d \pi_1}(\theta) \right] \biggr\} 
\Biggr\} \\ = \B{E} \biggl\{ \int \pi_1(d \theta) \B{1} 
\biggl( \frac{d \rho_1}{d \pi_1}(\theta) > 0 \biggr)  
\exp \bigl[ X(\theta) \bigr] \biggr\} 
\\ \leq \B{E} \Bigl\{ \tint \pi_1(d \theta) \exp \bigl[  
X(\theta) \bigr] \Bigr\} = 
\tint \pi_1(\theta) \B{E} \Bigl\{ \exp \bigl[ X (\theta) \bigr] \Bigr\}
= 1.
\end{multline*}
We can now use the fact that $\B{P} \rho_1$ is the joint distribution 
of the sample and of $\ov{\theta}_2$ and 
Chebyshev's exponential inequality, to prove with probability 
at least $1 - \epsilon_2$ that
$$
X(\ov{\theta}_2) \leq \log(\epsilon_2^{-1}) + \log \bigl(1 + x_2^{-1} \bigr).
$$
\newpage
\begin{align*}
\text{As} \quad \wh{\theta}_{\alpha_2}(\ov{\theta}_2) 
& \leq \ov{\theta}_2 + \frac{X(\ov{\theta}_2)}{n \alpha_2} 
\\ & \qquad \qquad + \frac{1}{\alpha_2} 
\log \Biggl( 1 + \alpha_2(m - \ov{\theta}_2) 
+ \frac{\alpha_2^2}{2} \bigl[ v + (m - \ov{\theta}_2)^2 \bigr] \Biggr)
\\ & \leq m + \frac{\alpha_2}{2} \bigl[ v + (m - \ov{\theta}_2)^2 \bigr] 
+ \frac{X(\ov{\theta}_2)}{n \alpha_2},
\end{align*}
we deduce that with probability at least $1 - \epsilon_2$, 
$$
\wh{\theta}_{\alpha_2}(\ov{\theta}_2) - m  \leq 
\frac{\alpha_2}{2} \bigl[ v + (m - \ov{\theta}_2)^2 \bigr] 
+ \frac{\log(\epsilon_2^{-1}) + \log(1 + x_2^{-1})}{n \alpha_2}
\leq \delta_2.
$$
We can prove in the same way that with probability at least $1 - 
\epsilon_2$, 
$$
m - \wh{\theta}_{\alpha_2}(\ov{\theta}_2) \leq 
\delta_2.
$$
We deduce that with probability at least $1 - 2 \epsilon_2$, 
$$
\lvert m - \wh{\theta}_{\alpha_2}(\ov{\theta}_2) \rvert
\leq \delta_2.
$$
Moreover, we see from the definition of $\ov{\theta}_2$ that 
with probability at least $1 - 2 \epsilon_1$, 
$\wt{\theta}_2 = \wh{\theta}_{\alpha_2}(\ov{\theta}_2)$, therefore with probability at least
$1 - 2 (\epsilon_1 + \epsilon_2)$, 
$$
\lvert m - \wt{\theta}_2 \rvert \leq \delta_2.
$$

The induction carries on in the same way. Assuming that with 
probability at least $1 - 2 \sum_{i=1}^{k-1} \epsilon_i$, 
$\lvert m - \wt{\theta}_{k-1} \rvert \leq \delta_{k-1}$, 
we deduce that with probability at least $1 - 2 \sum_{i=1}^k 
\epsilon_i$, $\lvert m - \wt{\theta}_k \rvert \leq \delta_k$.

\subsection{Proof of Proposition \thmref{prop3.1.2}}

The proof is the same as the previous one, except for the first
step, which is a consequence of the Chebyshev inequality, applied 
to the second moment of the empirical mean:
$$
\B{P} \Bigl( \lvert \wt{\theta}_1 - m \rvert \geq \delta_1 \Bigr) 
\leq \frac{\B{E} \bigl[ (\wt{\theta}_1 - m)^2 \bigr]}{\delta_1^2} 
\leq 2 \epsilon_1.
$$

\subsection{Proof of Proposition \thmref{prop2.1.1}}

Jensen's inequality for convex functions can serve to pull 
the integration with respect to $\rho_{\theta_0}$ out 
of the logarithm. Using moreover Equation \myeq{eq9.1.2}, 
we get the following chain of inequalities:

\begin{multline*}
- n \alpha M_{\alpha}(\theta_0) \leq - \tint \rho_{\theta_0}(d \theta) 
\sum_{i=1}^n \log \biggl[ 1 - \alpha \bigl(\theta - Y_i
\bigr) + \frac{\alpha^2}{2} 
\bigl( \theta - Y_i \bigr)^2 \biggr] \\ \leq  
\tint \rho_{\theta_0}(d \theta) \sum_{i=1}^n 
\log \biggl[ 1 + \alpha \bigl( \theta - Y_i \bigr) + \frac{\alpha^2}{2} 
\bigl(\theta - Y_i \bigr)^2 \biggr]. 
\end{multline*}
To proceed, let us consider the empirical process 
$$
W(\theta) = \sum_{i=1}^n \log \biggl[ 1 + \alpha \bigl( \theta - 
Y_i \bigr) + \frac{\alpha^2}{2} \bigl( \theta - Y_i \bigr)^2 \biggr].
$$
It satisfies
\begin{multline*}
\B{E} \Bigl\{ \exp \bigl[ W(\theta) \bigr] 
\Bigr\} = 
\B{E} \biggl\{ \prod_{i=1}^n \biggl[ 1 + \alpha \bigl(\theta - Y_i
\bigr) + \frac{\alpha^2}{2} \bigl( \theta - Y_i \bigr)^2 \biggr] \biggr\}
\\ = \biggl\{ 1 + \alpha(\theta - m) + \frac{\alpha^2}{2} 
\biggl[ \bigl( 
\theta - m \bigr)^2 + \B{E}\bigl[ (Y-m)^2 \bigr] \biggr] \biggr\}^n
\\ \leq \biggl\{ 1 + \alpha(\theta-m) + \frac{\alpha^2}{2} 
\biggl[ (\theta - m)^2 + v \biggr] \biggr\}^n.
\end{multline*}
Thus if we put
$$
w(\theta) = n \log \biggl\{ 1 + \alpha (\theta - m) 
+ \frac{\alpha^2}{2} \biggl[ (\theta - m)^2 + v \biggr] \biggr\}
$$
we see that 
$$
\B{E} \Bigl\{ \exp \bigl[ W(\theta) - w(\theta) \bigr] \Bigr\} \leq 1,
$$
(with equality when $\B{E}\bigl[ (Y-m)^2 \bigr] = v$). 
We can then follow the usual PAC-Bayesian route, choosing as reference
measure $\rho_{m}$. This consists in the inequalities
\begin{multline*}
\B{E} \biggl\{ \exp \biggl[ \sup_{\theta_0 \in \B{R}}  
\tint \rho_{\theta_0} ( d \theta) \bigl[ 
W(\theta) - w(\theta) \bigr] - \C{K} (\rho_{\theta_0}, \rho_m) 
\biggr] \biggr\} \\ \leq \B{E} \biggl\{ \tint \rho_m(d \theta) \exp \Bigl[ W(\theta) - w(\theta) 
\Bigr] \biggr\} \\ = 
\tint \rho_m(d \theta) \B{E} \biggl\{ \exp \biggl[ W(\theta) - w(\theta) 
\biggr] \biggr\} = 1, 
\end{multline*}
where we have used Fubini's theorem and the convex inequality
\begin{multline*}
\sup_{\rho \in \C{M}_+^1(\B{R})} \tint \rho( d \theta) 
\bigl[ W(\theta) - w(\theta) \bigr] - \C{K}(\rho, \rho_m) \\ 
= 
\log \Bigl\{ \tint \rho_m(d \theta) \exp \bigl[ W(\theta) - w(\theta) \bigr]  
\Bigr\}.
\end{multline*}
(See \cite[page 159]{Cat01} for a proof.)

From the exponential Chebyshev inequality 
$\B{P}\bigl[ X \geq \eta \bigr] \leq 
\B{E}\bigl[ \exp(X - \eta) \bigr]$, it follows that with probability
at least $1 - \epsilon$, for any $\theta_0 \in \B{R}$, 
$$
\tint \rho_{\theta_0}(d \theta) W(\theta) 
\leq \tint \rho_{\theta_0}(d \theta) w(\theta) + \C{K}(\rho_{\theta_0}, 
\rho_m ) - \log(\epsilon). 
$$
We can then remark that $\C{K}(\rho_{\theta_0}, \rho_m) 
= \frac{n \beta \alpha^2}{2} (\theta_0 - m)^2$ and that 
\begin{multline*}
\tint \rho_{\theta_0}(d \theta) 
w(\theta) \leq n \tint \rho_{\theta_0}(d \theta) 
\biggl\{ \alpha( \theta - m) + \frac{\alpha^2}{2} \bigl[ 
(\theta - m)^2 + v \bigr] \biggr\} 
\\ = n\alpha(\theta_0 - m) + \frac{n \alpha^2}{2} \bigl[ 
(\theta_0 - m)^2 + v \bigr] + \frac{1}{2 \beta},  
\end{multline*}
to conclude that with probability at least $1 - \epsilon$, 
for any $\theta_0 \in \B{R}$, 
\begin{multline*}
\tint \rho_{\theta_0}( d \theta) W(\theta) 
\leq n\alpha(\theta_0 - m) + \frac{n \alpha^2}{2} 
\Bigl[ (\theta_0 - m)^2 + v \Bigr] \\ + \frac{1}{2 \beta} +  
\frac{n \beta \alpha^2}{2} 
(\theta_0 - m)^2 - \log( \epsilon).
\end{multline*}
As we have already established that $- n \alpha M_{\alpha}( \theta_0) 
\leq \tint \rho_{\theta_0}(d \theta) W(\theta)$, this completes
the proof of the proposition.

\subsection{Proof of Proposition \thmref{prop1.2}}

It is straightforward to realize that 
$$
\B{E} \Bigl\{ \exp \bigl[ n \alpha M_{\alpha}(\theta_0) 
\bigr] \Bigr\} \leq \biggl\{ 1 - \alpha(\theta_0 - m) 
+ \frac{\alpha^2}{2} \biggl[(\theta_0 - m)^2 + v 
\biggr] + \frac{1}{2 n \beta} \biggr\}^n.
$$
The result then follows as in the previous proof from the exponential
Chebyshev inequality.

\subsection{Proof of Proposition \thmref{prop1.4}}
We will need the following 
elementary lemma.
\begin{lemma}
For any positive real constants $a$ and $c$ such 
that $4 a c \leq 1$, 
\begin{multline*}
\bigl\{ x \in \B{R} : x > a x^2 + c \bigr\} 
\\ = \biggl) \frac{2 c}{1 + \sqrt{1 - 4ac}}, 
\frac{ 1 + \sqrt{1 - 4 ac}}{2a}\biggr( 
\\ \supset \biggl( \frac{c}{1 - 2 ac}, 
\frac{1 - 2ac}{a} \biggr). 
\end{multline*}
\end{lemma}
Using Proposition \thmref{prop2.1.1} and the previous lemma, 
we see that with probability $1 - \epsilon_2$
\begin{align*}
m & \leq \wh{\theta}_{\alpha} + \frac{2}{(1+\beta)\alpha} 
\varphi \Biggl( 
\frac{(1+\beta)\bigl[n \alpha^2 v + \beta^{-1} - 2 \log(\epsilon_2) \bigr]}{
4n} \Biggr) \\ 
\text{or } \quad
m & \geq  
\wh{\theta}_\alpha + \frac{\bigl[n \alpha^2 v + \beta^{-1} + 2 \log(\epsilon_2^{
-1})\bigr]}{2n\alpha} \\
& \qquad \times \varphi \left( \frac{(1+\beta)\bigl[ n \alpha^2 v 
+ \beta^{-1} + 2 \log( \epsilon_2^{-1}) \bigr]}{4n} \right)^{-1}
\\ & 
 \geq \wh{\theta}_{\alpha} +
\frac{2}{(1+\beta)\alpha} \Biggl( 
1 - \frac{(1+\beta) \bigl[ 
n \alpha^2 v + \beta^{-1} - 2 \log(\epsilon_2) \bigr]}{2n} 
\Biggr).
\end{align*}
Let us make sure that the second condition cannot be fulfilled 
when $\lvert \wt{\theta}_1 - m \rvert \leq \delta_1$, assuming
that 
$$
\epsilon_2 > \exp \Biggl\{ 
- n \biggl[\frac{1}{1+\beta} - \alpha \delta_1 
- \frac{\bigl(n \alpha^2 v + \beta^{-1} \bigr)}{2n} 
\biggr] \Biggr\}, 
$$
or more accurately that 
\begin{multline*}
4 n \alpha \delta_1 < \bigl[ n \alpha^2 v + \beta^{-1} 
+ 2 \log(\epsilon_2^{-1}) \bigr]  \\ 
\times \varphi \left( \frac{(1+\beta) \bigl[ n \alpha^2 v + \beta^{-1} 
+ 2 \log(\epsilon_2^{-1}) \bigr]}{4 n} \right)^{-1}.
\end{multline*}
In this case, with probability at least $1 - \epsilon_2$, 
either $\lvert \wt{\theta}_1 - m \rvert > \delta_1$ or 
$$
m \leq \wh{\theta}_{\alpha} + \frac{2}{(1+\beta)\alpha} 
\varphi \Biggl( 
\frac{(1+\beta)\bigl[n \alpha^2 v + \beta^{-1} - 2 \log(\epsilon) \bigr]}{
4n} \Biggr).
$$
On the other hand, let us consider 
$$
\theta_0 = m + \frac{2}{\alpha} \varphi \biggl( 
\frac{n \alpha^2 v + \beta^{-1} - 2 \log(\epsilon_2)}{4 n} \biggr).
$$
From Proposition \thmref{prop1.2}, with probability at least $1 - \epsilon_2$, 
$M(\theta_0) \leq 0$, and therefore
$$
\wh{\theta}_{\alpha} \leq 
\theta_0
\leq m + \frac{2}{(1+\beta)\alpha} \varphi \biggl( 
\frac{(1+\beta) \bigl[ n \alpha^2 v + \beta^{-1} 
- 2  \log(\epsilon_2) \bigr]}{4n}\biggr). 
$$
Thus with probability at least $1 - 2 \epsilon_2$, either
$\lvert \wt{\theta}_1 - m \rvert > \delta_1$ or 
$$
\lvert \wh{\theta}_\alpha - m \rvert \leq \frac{2}{(1+\beta)\alpha} 
\varphi \left( \frac{(1+\beta) \bigl[ n \alpha^2 v + \beta^{-1} 
- 2 \log(\epsilon_2) \bigr]}{4n} \right).
$$
This proves the first part of the proposition. 
The consequences drawn from special choices of $\alpha$ 
are obvious, except for the last condition which may require 
some verification: when 
$\ds \alpha = \left( \frac{\beta^{-1} - 2 \log(\epsilon_2)}{nv} \right)^{1/2}$, 
putting $\gamma = \beta^{-1} - 2 \log(\epsilon_2)$ 
condition \myeq{eq1.1} becomes
$$
\delta_1 \leq \frac{1}{(1+\beta)} \left( \frac{nv}{\gamma} \right)^{1/2} 
\left( 1 - \frac{(1+\beta) \gamma}{n} \right).
$$
This can also be written as
$$
\frac{(1+\beta)}{n} \gamma + \frac{(1+\beta) \delta_1}{\sqrt{nv}} \sqrt{\gamma} 
- 1 \leq 0,
$$
which is a second order inequality in $\sqrt{\gamma}$.
Considering that $\sqrt{\gamma} \geq 0$, its solution is
$$
\sqrt{\gamma} \leq \frac{2}{\ds \frac{(1+\beta) \delta_1}{\sqrt{nv}} 
+ \sqrt{\frac{(1+\beta)^2 \delta_1^2}{nv} + \frac{4(1+\beta)}{n}}}. 
$$ 
To simplify formulas, we can remark that this inequality is 
satisfied when 
$$
\sqrt{\gamma} \leq \left( \frac{(1+\beta)^2 \delta_1^2}{nv} + \frac{4(1+\beta)}{n} 
\right)^{-1/2},
$$
that is when
$$
\epsilon_2 \geq \exp \left( \frac{1}{2\beta} - 
\frac{n v}{\ds 2 (1+\beta)^2 \delta_1^2 + 8 (1+\beta) v} \right).
$$

\subsection{Proof of Lemma \thmref{lemma5.1}} 

Using the fact that the $L_2$ norm of a sum is less than
the sum of the norms, and the definition of the kurtosis,
we get
\begin{multline*}
\B{E} \bigl( Y^4 \bigr) = \B{E} \Bigl\{ 
\bigl[ (Y - m)^2 +  m(2Y-m) \bigr]^2 \Bigr\}
\\ \leq  \biggl\{ \Bigl\{ \B{E} \bigl[ (Y-m)^4 \bigr]^{1/2} 
+ \lvert m \rvert \B{E} \bigl[ (2 Y - m)^2 \bigr]^{1/2} \biggr\}^2
\\ \leq \biggl\{ \kappa^{1/2} \B{E} \bigl[(Y- m)^2\bigr] 
+ \lvert m \rvert \B{E} \bigl\{ \bigl[ 2 (Y-m) + m \bigr]^2 \bigr\}^{1/2} \biggr\}^2
\\ = \biggl\{ \kappa^{1/2} \bigl(\B{E}(Y^2) - m^2) + \lvert m \rvert \bigl( 
4\B{E}(Y^2) -3m^2\bigr)^{1/2} \biggr\}^2.
\end{multline*}
Introducing $\ds y = \frac{m^2}{\B{E}(Y^2)}$, this gives 
\hfill $\ds
\frac{\B{E}(Y^4)}{\B{E}(Y^2)^2} \leq 
\biggl( \kappa^{1/2} (1 - y) + y^{1/2} (4 - 3 y)^{1/2} \biggr)^2.
$
Let us consider the function $f : (0,1) \mapsto \B{R}$ defined
as 
$$
f(y) = \kappa^{1/2} (1 - y) + y^{1/2}(4 - 3 y)^{1/2}.
$$
It reaches its maximum at point $x$ satisfying $ f'(x) = 0 $,
that is
$$
- \kappa^{1/2} + \frac{1}{2} x^{-1/2}(4-3x)^{1/2} 
- \frac{3}{2} x^{1/2} (4 - 3x)^{-1/2} = 0.
$$
Therefore $x$ satisfies
$\ds
\kappa x(4-3x) = (2-3x)^2$ or $\ds 
3x^2 - 4 x + \frac{4}{\kappa+3} = 0$.
Thus $\ds x = \frac{2}{3} \biggl( 1 - \sqrt{\frac{\kappa}{\kappa+3}} \biggr)$,
and 
$$
\sup_{y \in (0,1)} f(y) = \kappa^{1/2} 
\biggl(\frac{1}{3} + \frac{2}{3} \sqrt{\frac{\kappa}{\kappa+3}} 
\biggr) + 2 \sqrt{\frac{1}{\kappa+3}} = 
\frac{2}{3} \sqrt{\kappa + 3} + \frac{1}{3} \sqrt{\kappa}.
$$
This proves that 
\begin{multline*}
\frac{\B{E}(Y^4)}{\B{E}(Y^2)^2} \leq \frac{1}{9} \Bigl( 
\sqrt{\kappa} + 2 \sqrt{\kappa + 3}\Bigr)^{2} = 
\frac{\kappa}{9} \Biggl( 5 + \frac{12}{\kappa} + 
4\sqrt{1 + \frac{3}{\kappa}} \Biggr) 
\\ \leq \frac{\kappa}{9} \biggl( 5 + \frac{12}{\kappa} 
+ 4 + \frac{6}{\kappa} \Biggr) = \kappa + 2.
\end{multline*}
The same is of course true for $Y-\theta$ for any shift $\theta$,
as a mere change of notations shows, proving the first assertion 
of the lemma. 

Consider now for the lower bound the Bernoulli distribution 
with parameter $p$. In this case
\begin{align*}
\B{E} \bigl[ (Y-m)^2 \bigr] & = p (1-p)^2 + (1-p)p^2 
= p(1-p),\\
\B{E} \bigl[ (Y-m)^4 \bigr] & = p (1-p)^4 + (1-p)p^4 
= p(1-p)(1 - 3p + 3p^2),
\end{align*}
thus 
$$
\kappa = \frac{1 - 3p + 3p^2}{p(1-p)} = 
p^{-1} - 2 + \frac{p}{1-p}.
$$
Moreover $\ds \frac{\B{E}(Y^4)}{\B{E}(Y^2)^2} = p^{-1} \leq c$, 
and thus
$$
c - \kappa \geq 2 - \frac{p}{1-p}.
$$
While $p$ tends to zero, 
this proves that $\ds \sup_{\B{P} \in \C{M}(\B{R})_+^1} 
c_{\B{P}} - \kappa_{\B{P}} 
\geq 2$, and therefore, due to the already proved upper bound, that
$\ds \sup_{\B{P} \in \C{M}_+^1(\B{R})} c_{\B{P}} - \kappa_{\B{P}} = 2$.

When the skewness is null, that is when $\B{E}\bigl[ (Y-m)^3 \bigr]= 0$, 
we can write, assuming without loss of generality that $m = 0$,
\begin{multline*}
\B{E}\bigl[ (Y + \theta )^4 \bigr] = 
\B{E}\bigl(Y^4 \bigr) + 6 \theta^2 \B{E} \bigl( Y^2 \bigr) 
+ \theta^4 \\ = \kappa \B{E} \bigl( Y^2 \bigr)^2 
+ 6 \theta^2 \B{E} \bigl( Y^2 \bigr) + \theta^4 
\\ = \kappa \Bigl\{ \B{E}\bigl[ (Y+\theta)^2 \bigr]  - \theta^2 \Bigr\}^2 
+ 6 \theta^2 \Bigl\{ \B{E}\bigl[ (Y+\theta)^2\bigr]  - \theta^2 \Bigr\} 
+ \theta^4.
\end{multline*}
Thus, introducing $y = \frac{\theta^2}{\B{E}[(Y+\theta)^2]}$, 
we see that
\begin{multline*}
c = \sup_{\theta \in \B{R}} \frac{\B{E}\bigl[ (Y+\theta)^4\bigr] }{
\B{E}\bigl[ (Y+\theta)^2\bigr]^2} = 
\sup_{y \in (0,1)} \kappa(1-y)^2 + 6y(1-y) + y^2
\\ = \sup_{y \in (0,1)} \kappa - 2 (\kappa - 3) y + (\kappa - 5) y^2
= 
\begin{cases}
\ds \kappa + \frac{(3-\kappa)^2}{(5 - \kappa)}, & 1 \leq \kappa \leq 3,\\
\ds \kappa, & \kappa \geq 3.
\end{cases}
\end{multline*}

\subsection{Proof of Proposition \thmref{prop5.2}}

We are going to build a non observable variant of the construction 
made in the proposition, for which the conclusions of the proposition 
are always fulfilled because we enforced them.

Remember that the sequences $\delta_i$, $\gamma_i$ and $\zeta_{2i-1}$ 
take non random values, and let us define 
\begin{align*}
\overline{\theta}_1 & = \theta_1,\\
\overline{q}_1 & = 
\begin{cases}
\ds 
\frac{\delta_1 \exp( - \zeta_1)}{
Q_{\ov{\theta}_1, \delta_1}^{-1} \bigl[-(c-1) \delta_1^2\bigr]},
& \text{when } \ov{q}_1 \text{ thus defined satisfies } \\ & \lvert \log(\ov{q}_1) - \log \bigl[ 
v + (m - \ov{\theta}_1)^2 \bigr] \rvert \leq \zeta_1,\\[1ex]
\ds v + (m - \ov{\theta}_1)^2, & \text{otherwise,}\\
\end{cases}\\
\ov{q}_2 & = \ov{q}_1 \exp( x_2 \zeta_1 U_2),\\
\ov{\alpha}_2 & = \sqrt{\frac{2 \bigl[ \log(\epsilon_2^{-1}) + \gamma_2 
\bigr]}{n \ov{q}_2}},\\
\ov{\zeta}_2 & = \exp \biggl[ \frac{(1+x_2) 
\zeta_1}{2} \biggr] 
\sqrt{\frac{2 \ov{q}_2 \bigl[ \log(\epsilon_2^{-1}) + \gamma_2 \bigr]}{
n}},\\
\ov{\theta}_2 & = \begin{cases}
\wh{\theta}_{\ov{\alpha}_2}(\ov{\theta}_1), & \text{when } 
\ov{\theta}_2 \text{ thus defined satisfies } \\ & \lvert m - \ov{\theta}_2 
\rvert \leq \ov{\zeta}_2,\\
m, & \text{otherwise,}
\end{cases}\\
& \vdots \\
\ov{\theta}_{2i-1} & = \ov{\theta}_{2i-2} + 
\ov{\zeta}_{2i-2} x_{2i-1} U_{2i-1},\\
\ov{q}_{2i-1} & = 
\begin{cases}
\ds \frac{\delta_{2i-1} \exp( - \zeta_{2i-1})}{
Q_{\ov{\theta}_{2i-1}, \delta_{2i-1}} \bigl[ 
- (c-1) \delta_{2i-1}^2 \bigr]},  \text{ when } \ov{q}_{2i-1} \text{ thus
defined satisfies } \\[2ex] \ds ~ \qquad \qquad \qquad \lvert \log (\ov{q}_{2i-1}) 
- \log \bigl[ v + (m - \ov{\theta}_{2i-1})^2 \bigr] \rvert  \leq \zeta_{2i-1},\\[1ex]
v + (m - \ov{\theta}_{2i-1})^2,  \text{ otherwise,}
\end{cases}\\
\ov{q}_{2i} & = \ov{q}_{2i-1} \exp (x_{2i} \zeta_{2i-1} U_{2i}),\\
\ov{\alpha}_{2i} & = 
\exp \biggl[ - \frac{(1 + x_{2i}) \zeta_{2i-1}}{2} 
\biggr] \sqrt{\frac{2 \bigl[ \log(\epsilon_{2i}^{-1}) + 
\gamma_{2i} \bigr]}{n \ov{q}_{2i}}},\\
\ov{\zeta}_{2i} & = \exp \biggl[ \frac{(1+x_{2i}) \zeta_{2i-1}}{2} \biggr] 
\sqrt{\frac{2 \ov{q}_{2i} \bigl[ \log(\epsilon_{2i}^{-1}) + 
\gamma_{2i} \bigr]}{n}},\\
\ov{\theta}_{2i} & = 
\begin{cases} 
\wh{\theta}_{\ov{\alpha}_{2i}}(\ov{\theta}_{2i-1}),& \text{when } 
\ov{\theta}_{2i} \text{ thus defined satisfies}\\
& \lvert m - \ov{\theta}_{2i} \rvert \leq \ov{\zeta}_{2i}.\\
m, & \text{otherwise,}
\end{cases} \\
& \vdots
\end{align*}
By construction, these modified quantities are such that
for any $i =1, \dots, k$,
\begin{gather*}
\lvert m - \ov{\theta}_{2i} \rvert \leq \ov{\zeta}_{2i},\\
\bigl\lvert \log(\ov{q}_{2i-1}) - \log \bigl[ v + (m - \ov{\theta}_{2i-1})^2 
\bigr] \bigr\rvert \leq \zeta_{2i-1}.
\end{gather*}
Let us defined ``the modified sequence''
\begin{align*}
S_{2j-1} & = \Bigl\{ \bigl(\ov{\theta}_{2i-1}\bigr)_{i=2}^j, \bigl[ 
\log (\ov{q}_{2i}) \bigr]_{i=1}^{j-1} \Bigr\} = S_{2j-2} 
\cup \{ \ov{\theta}_{2K-1} \},\\
S_{2j} & = \Bigl\{ \bigl(\ov{\theta}_{2i-1}\bigr)_{i=2}^j, \bigl[ 
\log (\ov{q}_{2i}) \bigr]_{i=1}^{j} \Bigr\} = S_{2j-1} \cup \{ 
\log (\ov{q}_{2j}) \}.
\end{align*}
The first step of the proof will be to prove by induction on $j$ 
the following lemma. 

\begin{lemma}
There exists some prior distribution $\pi_{j}$ on the modified sequence
$S_j$ (that is some non random probability measure on $\B{R}^{j-1}$) such 
that the joint conditional distribution $\rho_{j}$ of 
the modified sequence $S_j$ knowing the sample $(Y_i)_{i=1}^n$
is such that $\ds \log \biggl( \frac{d \rho_{j}}{d \pi_{j}} \biggr) \leq \gamma_{j}$.
\end{lemma}

\begin{proof}
Indeed, assuming that this is true for $2j-2$, we build $\pi_{2j-1}$ 
and $\pi_{2j}$ 
from $\pi_{2j-2}$ by deciding that $\pi_{2j-2}$ is the marginal 
of $\pi_{2j-1}$ on $S_{2j-2}$ and that $\pi_{2j-1}$ is the 
marginal of $\pi_{2j}$ on $S_{2j-1}$. 
We complete the definition of $\pi_{2j-1}$ by defining the conditional 
distribution of $\ov{\theta}_{2j-1}$ 
knowing $S_{2j-2}$ {\em under $\pi_{2j-1}$} 
as the uniform probability distribution on 
the interval 
$$
m + (1+x_{2j-1})\ov{\zeta}_{2j-2}\times(-1,+1).
$$ 
Similarly we complete the definition of $\pi_{2j}$ by defining 
the conditional distribution of $\log(\ov{q}_{2j})$ knowing 
$S_{2j-2} \cup \{ \ov{\theta}_{2j-1} \}$
as the uniform probability distribution on the interval
$$
\log \bigl[ v + 
(m - \ov{\theta}_{2j-1})^2 \bigl] + (1+x_{2j}) \zeta_{2j-1} 
\times (-1,+1).
$$
As the conditional distribution of $\ov{\theta}_{2j-1}$ 
and $\log ( \ov{q}_{2j})$ knowing $S_{2j-2}$ and the sample
$(Y_i)_{i=1}^n$ is the product of the uniform probability 
measure on the interval 
$$
\ov{\theta}_{2j-2} + \ov{\zeta}_{2j-2} x_{2j-1}
\times (-1,+1)
$$
and the uniform probability measure on the interval
$$
\log(\ov{q}_{2j-1}) + \zeta_{2j-1} x_{2j} \times (-1, +1),
$$
it is readily seen that
$\ds
\frac{d \rho_{2j-1}( \cdot | S_{2j-2})}{d \pi_{2j-1}( \cdot | S_{2j-2})}  
= 1 + x_{2j-1}^{-1}$ on 
the support of $\rho_{2j-1}( \cdot | S_{2j-2})$. Using the 
induction hypothesis we deduce that 
$$
\frac{d \rho_{2j-1}}{d \pi_{2j-1}} = \frac{d \rho_{2j-2}}{d \pi_{2j-2}} 
\times \frac{d \rho_{2j-1}( \cdot | S_{2j-2})}{d \pi_{2j-1}(\cdot | 
S_{2j-2})} \leq \exp (\gamma_{2j-2}) (1 + x_{2j-1}^{-1}) 
= \exp( \gamma_{2j-1}).
$$
We deduce in the same way that 
$$
\frac{d \rho_{2j}}{d \pi_{2j}} = \frac{d \rho_{2j-1}}{d \pi_{2j-1}} 
\times \frac{ d \rho_{2j}(\cdot|S_{2j-1} )}{ 
d \pi_{2j}(\cdot|S_{2j-1})} 
\leq \exp( \gamma_{2j}).
$$
Moreover the first step is easy to prove, taking for 
$\pi_1$ the uniform probability measure on the interval
$$
\log \bigl[ v + (m - \theta_1)^2 \bigr] + \zeta_1 x_2 \times (-1, +1).
$$
This achieves to prove the lemma by induction.
\end{proof}
Let us now proceed with a second lemma.

\begin{lemma}
\label{lemma9.4}
With probability at least $1 - 2 \epsilon_{2i-1}$, 
$$
\ov{q}_{2i-1} = \frac{\delta_{2i-1} \exp(- \zeta_{2i-1})}{Q_{\ov{\theta}_{2i-1}, 
\delta_{2i-1}}^{-1} \bigl[ -(c-1) \delta_{2i-1}^2 \bigr]}.
$$
\end{lemma}

\begin{proof}
Let us remark that
\begin{multline*}
\B{E} \Bigl\{ \exp \bigl[ n Q_{\theta, \delta_{2i-1}}
(\alpha)  \bigr] \Bigr\}
\leq \exp \biggl\{ n \alpha \bigl[v + (\theta - m)^2 \bigr] 
- \delta_{2i-1} \\ + \frac{n}{2} \B{E} 
\Bigl\{ \alpha(Y_i - \theta)^2 - \alpha \bigl[ v 
+ (\theta - m)^2 \bigr] \Bigr\}^2 
+ \frac{n}{2} \Bigl\{ \alpha \bigl[ v + (\theta - m)^2 \bigr]  - 
\delta_{2i-1} \Bigr\}^2 \biggr\}
\\ \leq  \exp \bigl[ n g(\theta) \bigr],
\end{multline*}
where
\begin{multline*}
g(\theta) = \alpha \bigl[v + (\theta - m)^2 \bigr] - \delta_{2i-1} 
+ \frac{(c-1) \alpha^2}{2} \bigl[ v + (\theta - m)^2 \bigr]^2 
\\ + \frac{1}{2} \Bigl\{ \alpha \bigl[ v + (\theta - m)^2 \bigr]
- \delta_{2i-1} \Bigr\}^2.
\end{multline*}
Integrating the previous exponential moment with respect to 
$\rho_{2i-1}$, and taking expectations with respect to 
the distribution of the sample $(Y_i)_{i=1}^n$,
we get, for any measurable mapping $S_{2i-1} \mapsto \alpha(S_{2i-1})$ 
to be chosen afterward,
\begin{multline*}
\B{E} \Bigl\{ \tint \rho_{2i-1} (d S_{2i-1})  
\exp \Bigl[ n Q_{\ov{\theta}_{2i-1}, \delta_{2i-1}}(\alpha) - 
ng(\ov{\theta}_{2i-1}) - 
\gamma_{2i-1} \Bigr] 
\Bigr\}
\\ \shoveleft{\quad \leq \B{E} \Bigl\{ \tint \rho_{2i-1} (d S_{2i-1})  
\B{1} \Bigl( \frac{ d \rho_{2i-1}}{d \pi_{2i-1}} 
> 0 \Bigr)} \\ \times \exp \Bigl[ n Q_{\ov{\theta}_{2i-1}, \delta_{2i-1}}(\alpha) - ng(\ov{\theta}_{2i-1}) - 
\log \Bigl( \tfrac{d \rho_{2i-1}}{d \pi_{2i-1}} \Bigr) \Bigr] 
\Bigr\} \\
\leq \B{E} \Bigl\{ \tint \pi_{2i-1}(d S_{2i-1}) \exp 
\Bigl[ n Q_{\ov{\theta}_{2i-1}, \delta_{2i-1}}(\alpha) 
- n g(\ov{\theta}_{2i-1}) \Bigr] \Bigr\}
\\ = \tint \pi_{2i-1}(d S_{2i-1}) \B{E} \Bigl\{ 
\exp \Bigl[ n Q_{\ov{\theta}_{2i-1}, \delta_{2i-1}}(\alpha) - n 
g(\ov{\theta}_{2i-1})  \Bigr] 
\Bigr\} \leq 1. 
\end{multline*}
(Let us remember that $\ov{\theta}_{2i-1}$ is the last component 
of $S_{2i-1}$. We made some dependences explicit, but not all of 
them, and more specifically the dependence of $\alpha$ here has been 
kept hidden.)

Using Chebyshev's exponential inequality, we deduce from 
this moment inequality that, for any measurable 
mapping $\ov{\theta}_{2i-1} \mapsto \alpha(\ov{\theta}_{2i-1})$, 
(that is for any choice of $\alpha$ which 
may depend on the value of $\ov{\theta}_{2i-1}$), 
with probability at least $1 - \epsilon_{2i-1}$, 
$$
Q_{\ov{\theta}_{2i-1}, \delta_{2i-1}} \bigl[ \alpha(\ov{\theta}_{2i-1}) 
\bigr] \leq g(\ov{\theta}_{2i-1})
+ \frac{\gamma_{2i-1} + \log(\epsilon_{2i-1}^{-1})}{n}.
$$

It is useful at this point to realize that the mapping  $\alpha \mapsto Q_{\theta, \delta}(\alpha)$
is increasing for any $\theta \in \B{R}$ 
and $\delta \in )0,1)$, as its derivative shows
$$
Q'_{\theta, \delta}(\alpha) = \frac{1}{n} \sum_{i=1}^n
\frac{(1-\delta)(Y_i-\theta)^2 + \alpha (Y_i-\theta)^4}{1 
+ \alpha(Y_i-\theta)^2 + \frac{1}{2} \bigl[ (Y_i-\theta)^2 - 
\delta \bigr]^2} \geq 0.
$$ 
In order to choose $\alpha$ (which is allowed to depend on
$\ov{\theta}_{2i-1}$), let us introduce temporarily $y = \alpha \bigl[ v 
+ (\ov{\theta}_{2i-1} - m)^2 \bigr] - \delta_{2i-1}$.
We can rephrase what we just proved saying that with probability at 
least $1 - \epsilon_{2i-1}$,
\begin{multline*}
\frac{y + \delta_{2i-1}}{v + (\ov{\theta}_{2i-1} - m)^2} 
\leq Q_{\ov{\theta}_{2i-1},\delta_{2i-1}}^{-1} \biggl[
y + \frac{(c-1)(y + \delta_{2i-1})^2}{2}  
+ \frac{y^2}{2} + \frac{(c-1) \delta_{2i-1}^2}{2} \biggr]  
\\ = 
Q_{\ov{\theta}_{2i-1},\delta_{2i-1}}^{-1} \biggl[ 
\frac{c y^2}{2} + \bigl[(c-1) \delta_{2i-1} + 1 \bigr] y 
+ (c-1) \delta_{2i-1}^2 \biggr]. 
\end{multline*}
Let us choose $y$ such that the argument of $Q_{\ov{\theta}_{2i-1}, 
\delta_{2i-1}}^{-1}$ in this inequality is equal 
to $-(c-1) \delta_{2i-1}^2$. This requires that $y$ should satisfy
$$
\frac{c y^2}{2} + \bigl[ (c-1) \delta_{2i-1} + 1 \bigr] 
y + 2 (c-1) \delta_{2i-1}^2  = 0.
$$
This has (negative) real roots when 
$$
\delta_{2i-1} \leq \frac{1}{2 \sqrt{c(c-1)} - (c-1)},
$$
and it is elementary to check that the largest of these negative roots
is 
$$
y = - \delta_{2i-1} h \bigl[ \tfrac{c}{c-1}, (c-1) \delta_{2i-1} \bigr] 
= - \bigl[ 1 - \exp( - 2 \zeta_{2i-1} ) \bigr] \delta_{2i-1}. 
$$
Thus with probability at least $1 - \epsilon_{2i-1}$, 
\begin{equation}
\label{eq9.3}
\frac{\exp( - 2 \zeta_{2i-1} ) \delta_{2i-1}}{v + (\ov{\theta}_{2i-1} 
- m)^2} \leq Q_{\ov{\theta}_{2i-1}, \delta_{2i-1}}^{-1} \bigl[ - (c-1) 
\delta_{2i-1}^2 \bigr].
\end{equation}
To get the reverse inequality, we may notice, due to Equation 
\myeq{eq9.1.2} that 
$$
- Q_{\theta, \delta}(\alpha) \leq \frac{1}{n} 
\sum_{i=1}^n \log \biggl\{ 1 - \alpha(Y_i-\theta)^2 + \delta 
+ \frac{1}{2} \Bigl[ \alpha(Y_i - \theta)^2 - \delta \Bigr]^2 
\biggr\}.
$$
Consequently 
$$
\B{E} \Bigl\{ \exp \bigl[ - n Q_{\theta, \delta}(\alpha) \bigr] 
\Bigr\} \leq \exp \bigl[ n \ov{g}(\theta) \bigr],
$$
where 
\begin{multline*}
\ov{g}(\theta) = \exp \biggl\{ n \delta - n \alpha \bigl[v + (\theta - m)^2 \bigr] 
\\ + \frac{(c-1) \alpha^2}{2} \bigl[ v + (\theta - m)^2 \bigr]^2 
+ \frac{1}{2} \Bigl\{ \alpha \bigl[ v + (\theta -m)^2 \bigr] - 
\delta \Bigr\}^2 \biggr\}.
\end{multline*}
Thus, integrating with respect to $\pi_{2i-1}$ as previously, we deduce 
that with probability at least $1 - \epsilon_{2i-1}$,
$$
- \ov{g}(\ov{\theta}_{2i-1}) - \frac{\gamma_{2i-1} + 
\log(\epsilon_{2i-1}^{-1})}{n} =  
- \ov{g}(\ov{\theta}_{2i-1}) 
- \frac{(c-1)}{2} \delta_{2i-1}^2 \leq Q_{\ov{\theta}_{2i-1},\delta_{2i-1}}
(\alpha), 
$$
where the choice of $\alpha$ may depend on $\ov{\theta}_{2i-1}$. 
Choosing then 
$\ds \alpha = \frac{\delta_{2i-1}}{v + (\ov{\theta}_{2i-1} - m)^2}$,
we get that with probability at least $1 - \epsilon_{2i-1}$, 
\begin{equation}
\label{eq9.4}
Q_{\ov{\theta}_{2i-1}, \delta_{2i-1}}^{-1} \bigl[-(c-1) \delta_{2i-1}^2 
\bigr] \leq \frac{\delta_{2i-1}}{v + (\ov{\theta}_{2i-1} - m)^2}.
\end{equation}
Taking the union bound of inequalities \myeq{eq9.3} and \myeq{eq9.4}, 
we see that with probability at least $1 - 2 \epsilon_{2i-1}$, 
$$
\frac{\exp( - 2 \zeta_{2i-1}) \delta_{2i-1}}{v + (\ov{\theta}_{2i-1} 
- m)^2} \leq Q_{\ov{\theta}_{2i-1}, \delta_{2i-1}}^{-1} \bigl[ - (c-1) 
\delta_{2i-1}^2 \bigr] \leq \frac{\delta_{2i-1}}{v + (\ov{\theta}_{2i-1} 
-m)^2}.
$$
This can be rewritten as
$$
\left\lvert \log \bigl[ v + (\ov{\theta}_{2i-1} - m)^2 \bigr] 
- \log \biggl[ \frac{\delta_{2i-1} \exp( - \zeta_{2i-1}) }{
Q_{\ov{\theta}_{2i-1}, \delta_{2i-1}}^{-1} 
\bigl[ -(c-1) \delta_{2i-1}^2 \bigr] } \biggr] \right\rvert 
\leq \zeta_{2i-1}.
$$
Therefore, coming back to the definition of $\ov{q}_{2i-1}$, 
we see that, with probability at least $1 - 2 \epsilon_{2i-1}$,
$$
\ov{q}_{2i-1} = \frac{\delta_{2i-1} \exp( - \zeta_{2i-1})}{ 
Q_{\ov{\theta}_{2i-1}, \delta_{2i-1}}^{-1} \bigl[ - (c-1) 
\delta_{2i-1}^2 \bigr]}.
$$
\end{proof}

\begin{lemma}
\label{lemma9.5}
With probability at least $1 - 2 \epsilon_{2i}$, $\ov{\theta}_{2i} = 
\wh{\theta}_{\ov{\alpha}_{2i}}(\ov{\theta}_{2i-1})$.
\end{lemma}
\begin{proof}
We start from the exponential moment inequality
$$
\B{E} \bigl\{ \exp \bigl[ n \alpha \wh{\theta}_{\alpha}(\theta) 
\bigr] \bigr\} 
\leq \exp \bigl[ n \alpha g(\theta) \bigr],
\text{ where } g(\theta) =  m  + \frac{\alpha}{2} 
\bigl[ v + (m - \theta)^2 \bigr].
$$
Integrating with respect to $\pi_{2i}$, we get, 
choosing the parameter $\alpha$ to be $\ov{\alpha}_{2i}$, 
depending on $S_{2i}$ through $\ov{q}_{2i}$, 
\begin{multline*}
\B{E} \Bigl\{ \tint \rho_{2i}(d S_{2i}) \exp \Bigl[ 
n \ov{\alpha}_{2i} \bigl[ \wh{\theta}_{\ov{\alpha}_{2i}}(\ov{\theta}_{2i-1}) - 
g(\ov{\theta}_{2i-1}) \bigr] - \gamma_{2i} \Bigr] 
\\ \shoveleft{\leq \B{E} \Bigl\{ \tint  \rho_{2i}(d S_{2i}) 
\B{1}\Bigl(\frac{d \rho_{2i}}{d \pi_{2i}} > 0
\Bigr)} \\ \times  \exp \Bigl[ n \ov{\alpha}_{2i} \bigl[ \wh{\theta}_{
\ov{\alpha}_{2i}}(\ov{\theta}_{2i-1})
- g(\ov{\theta}_{2i-1}) \bigr] - \log \Bigl( 
\frac{ d \rho_{2i} }{d \pi_{2i}} \Bigr) \Bigr]  \Bigr\}
\\ \leq \B{E} \Bigl\{ \tint \pi_{2i}(d S_{2i}) 
\exp \Bigl[ n \ov{\alpha}_{2i} \bigl[ \wh{\theta}_{
\ov{\alpha}_{2i}}(\ov(\theta)_{2i-1}) - 
g(\ov{\theta}_{2i-1}) \bigr] \Bigr] \Bigr\} 
\\ = \tint \pi_{2i}(d S_{2i}) \B{E} \Bigl\{ \exp \Bigl[ 
n \ov{\alpha}_{2i} \bigl[ \wh{\theta}_{\ov{\alpha}_{2i}}(\ov{\theta}_{2i-1}) - 
g(\ov{\theta}_{2i-1}) \bigr] \Bigr] \Bigr\} \leq 1.
\end{multline*}
Thus, according to Chebyshev's inequality, with probability at least $1 - 
\epsilon_{2i}$, 
\begin{multline*}
\wh{\theta}_{\ov{\alpha}_{2i}}(\ov{\theta}_{2i-1}) 
\leq m + \frac{\ov{\alpha}_{2i}}{2} \bigl[ v + (m - \ov{\theta}_{2i-1})^2 
\bigr] + \frac{\gamma_{2i} + \log( \epsilon_{2i}^{-1})}{n \ov{\alpha}_{2i}}
\\ \leq m + 
\exp \biggl[ \frac{(1 + x_{2i}) \zeta_{2i-1}}{2} \biggr] \sqrt{ \frac{2 \ov{q}_{2i} \bigl[ 
\gamma_{2i} + \log(\epsilon_{2i}^{-1}) \bigr]}{n}} = m + \ov{\zeta}_{2i}.
\end{multline*}
In the same way, considering $\ov{\theta}_{2i-1} - Y_i$ instead 
of $Y_i - \ov{\theta}_{2i-1}$, we can prove with probability at least 
$1 - \epsilon_{2i}$ that 
$$
m \leq \wh{\theta}_{\ov{\alpha}_{2i}}(\ov{\theta}_{2i-1}) 
+ \ov{\zeta}_{2i}.
$$
A union bound argument then proves that with probability at least 
$1 - 2 \epsilon_{2i}$, 
$$
\bigl\lvert m - \wh{\theta}_{\ov{\alpha}_{2i}}(\ov{\theta}_{2i-1}) 
\bigr\rvert \leq \ov{\zeta}_{2i}.
$$
Coming back to the definition of $\ov{\theta}_{2i}$, 
we see that it means that with probability at least $1 - 2 \epsilon_{2i}$,
$\ov{\theta}_{2i} = \wh{\theta}_{\ov{\alpha}_{2i}}(\ov{\theta}_{2i-1})$.
\end{proof}

We can now take a union bound of Lemma \thmref{lemma9.4} and Lemma
\thmref{lemma9.5}, for $i =1, \dots, k$, to see that with probability 
at least $1 - 2 \sum_{i=1}^{2k} \epsilon_{i}$, the constructions
of $\ov{q}_{i}$ and $\ov{\theta}_i$ coincide with the definitions
of $\wt{q}_i$ and $\wt{\theta}_i$, and therefore that $\ov{\theta}_{i} 
= \wt{\theta}_i$ and $\ov{q}_i = \wt{q}_i$, $i = 1, \dots, 2k$. 
Consequently, with probability at least $1 - 2 \sum_{j=1}^{2k} 
\epsilon_{j}$, for any $i = 1, \dots, k$,
\begin{gather*}
\lvert m - \wt{\theta}_{2i} \rvert \leq \zeta_{2i},\\
\bigl\lvert \log \bigl[ v + (m - \wt{\theta}_{2i-1})^2 \bigr] 
- \log \bigl( \wt{q}_{2i-1} \bigr) \bigr\rvert \leq \zeta_{2i-1},
\end{gather*}
proving Proposition \thmref{prop5.2}.

\subsection{Proof of Proposition \thmref{prop6.1}}
Let us consider the function 
$$
g(x) \overset{\text{def}}{=} x - \frac{1}{2} \log \left(  \frac{1 + x + \frac{x^2}{2}}{
1 - x + \frac{x^2}{2}} \right), \qquad  x \in \B{R}.
$$
\mbox{} \hfill 
\includegraphics{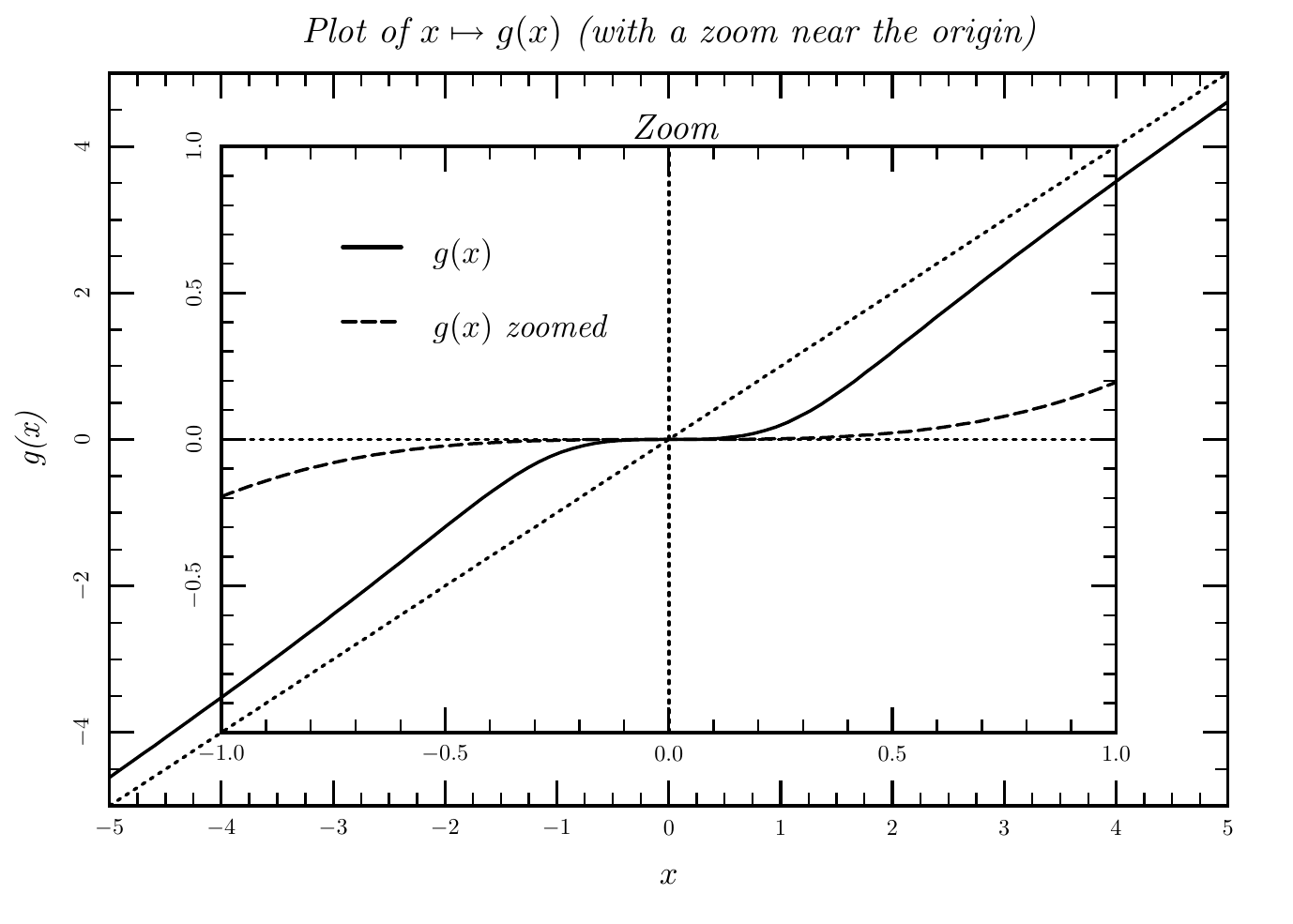} \hfill \mbox{}\\[-8ex]
\begin{lemma}
The function $g$ is bounded by
$$
\lvert g(x) \rvert  \leq \min \left\{ 
\frac{\lvert x \rvert^3}{5}, \frac{3 x^2}{10}, \lvert x \rvert \right\}, 
\qquad x \in \B{R}.
$$
\end{lemma}
The derivative of $g$ is
$$
g'(x) = \frac{x^2}{4} \left( \frac{1}{1+x+\frac{x^2}{2}} + 
\frac{1}{1 - x + \frac{x^2}{2}} \right) = 
\frac{x^2(2 + x^2)}{4 + x^4} \geq 0, \qquad x \in \B{R},
$$
showing that $g$ has the same sign as $x$. 
The fact that $\lvert g(x) \rvert \leq \lvert x \rvert$ is then clear 
from the sign 
of $\frac{1}{2} \log \left( \frac{1 + x + \frac{x^2}{2}}{1 - x 
+ \frac{x^2}{2}} \right)$, which is the same as the sign of $x$. 

Let us prove now that $\ds 
\lvert g(x) \rvert \leq \frac{\lvert x \rvert^3}{5}$. It is clearly 
enough to prove it for $x > 0$, because $\lvert g \rvert $ is symmetric. 
Let us consider $\ds h(x) = g(x) - \frac{x^3}{5}$ and let us compute
$$
h'(x) = \frac{x^2}{4+x^4} \biggl( - \frac{2}{5} + x^2 - \frac{3}{5} x^4 
\biggr) = - \frac{x^2 (x^2 - 1)(3 x^2 - 2)}{5(4 + x^4)}.
$$
From the sign of $h'$, we see that $h$ has a unique local maximum 
on the positive real line at point $x = 1$. Moreover
$h(1) = \frac{4}{5} - \frac{1}{2} \log(5) < 0$ (it is close 
to $-0.005$, as can be checked numerically). Thus 
$h(x) \leq 0$ for $x \in \B{R}_+$, implying that $\ds \lvert g(x) \rvert 
\leq \frac{\lvert x \rvert^3}{5}$ on the whole real line, 
as announced.

To prove $\ds \lvert g(x) \rvert \leq \frac{3 x^2}{10}$, 
we consider $\ds h_2(x) = g(x) - \frac{3 x^2}{10}$.  
A small computation shows that
$$
h_2'(x) = - \frac{x(x-1)(x-2)(3x^2 + 4x + 6)}{5(4 + x^4)}.
$$
Thus it has a unique local maximum on the positive real 
line at point $2$. Moreover $\ds h_2(2) = \frac{4}{5} - 
\frac{1}{2} \log(5) < 0$, showing that $\lvert g \rvert$ is 
upper-bounded by $\ds \frac{3 x^2}{10}$ on the positive real 
line, and therefore on the whole real line because it is 
symmetric. 

Let us remark now that with probability at least $1 - \epsilon$, 
\begin{multline*}
n \alpha (M - m) - \sum_{i=1}^n 
g\bigl[ \alpha(Y_i - m) \bigr] \\ \leq \sum_{i=1}^{n} \log \Bigl[  1 + \alpha ( Y_i - m) 
+ \tfrac{\alpha^2}{2} (Y_i - m)^2 \Bigr]  
\leq \frac{n \alpha^2}{2} v + \log(\epsilon^{-1}).
\end{multline*}
The first of these two inequalities comes from the fact that
$$
\frac{1}{2} \log \left( \frac{1 + x + \frac{x^2}{2}}{1 - x + \frac{x^2}{2}} 
\right) \leq \log \bigl( 1 + x + \tfrac{x^2}{2} \bigr), \qquad x \in \B{R}.
$$
In the same way, with probability at least $1 - \epsilon$, 
\begin{multline*}
n \alpha(M-m) - \sum_{i=1}^n g \bigl[ \alpha (Y_i - m) \bigr] 
\\ \geq - \sum_{i=1}^n \log \Bigl[ 1 - \alpha(Y_i - m) 
+ \tfrac{\alpha^2}{2} (Y_i - m)^2 \Bigr] \geq - \frac{n \alpha^2}{2} 
v - \log(\epsilon^{-1}).
\end{multline*}
Let us now deal with $\sum_{i=1}^n g \bigl[ \alpha(Y_i-m) \bigr]$.
We need some compact notations to manipulate this. 
Let $G_i = g\bigl[ \alpha(Y_i - m) \bigr]$ and $G = g \bigl[ 
\alpha (W-m) \bigr] $. 
Let us remark that 
$$
\lvert G  \rvert \leq \min \biggl\{ 
\frac{\alpha^3}{5}  \lvert Y - m \rvert^3, \frac{3 \alpha^2}{10} 
(Y - m)^2, \alpha \lvert Y - m \rvert \biggr\}.
$$
Moreover, using the fact that $\min \{ a,b \} 
\leq a^{2/3} b^{1/3}$, 
we see that 
$$
\lvert G \rvert \leq \left( \frac{
3}{10} \right)^{1/3} \alpha^{4/3} \lvert Y-m \rvert^{4/3}.
$$
With probability at least $1 - \epsilon$, 
\begin{multline*}
\biggl\lvert \frac{1}{n} \sum_{i=1}^n \bigl[  G_i  
- \B{E} (G)   \bigl]   \biggr\rvert 
\leq \epsilon^{-1/4} \B{E} \Biggl\{ 
\biggl[ \frac{1}{n} \sum_{i=1}^n \bigl[ G_i 
- \B{E}(G) \bigl] \biggr]^4 \Biggr\}^{1/4}
\\ =  
\bigl( \epsilon n^3 \bigr)^{-1/4} \Biggl\{ 
3 (n-1) \bigl[ 
\B{E}(G^2) - \B{E}(G)^2 \bigr]^2
+ \B{E} \Bigl\{ \bigl[  G - \B{E}(G) \bigr]^4 \Bigr\} 
\Biggr\}^{1/4} 
\\ 
= \bigl( \epsilon n^3 \bigr)^{-1/4} 
\Biggl\{ 3(n-1) \Bigl[ \B{E}(G^2)^2 
- 2 \B{E}(G^2)\B{E}(G)^2 + \B{E}(G)^4 \Bigl] 
\\ + \B{E}(G^4) - 4 \B{E}(G^3)\B{E}(G)  
+ 6 \B{E}(G^2)\B{E}(G)^2 - 3 \B{E}(G)^4 
\Biggr\}^{1/4}
\\ = \bigl( \epsilon n^3 \bigr)^{-1/4} 
\Biggl\{ 3(n-1) \B{E}(G^2)^2 
- 6(n-2) \B{E}(G^2)\B{E}(G)^2 
\\ + 3(n-2) \B{E}(G)^4 + \B{E}(G^4) - 4 \B{E}(G^3)\B{E}(G) 
\Biggr\}^{1/4}
\\ \leq \bigl( \epsilon n^3 \bigr)^{-1/4} 
\Biggl\{ 
3(n-1) \B{E}(G^2)^2 + \B{E}(G^4) + 4 \B{E}(\lvert G \rvert^3)
\B{E}(\lvert G \rvert)
\Biggr\}^{1/4} \\
\leq \bigl( \epsilon n^3 \bigr)^{-1/4} 
\Biggl\{ 3(n-1) \Bigl[ \tfrac{9}{100} \alpha^4 \B{E}
\bigl[ (Y-m)^4 \bigr]  \Bigr]^2 
+ \alpha^4 \B{E}\bigl[ (Y-m)^4 \bigr] 
\\ + \tfrac{6}{25} \alpha^7 \B{E} \bigl[ (Y-m)^4 \bigr] \B{E} 
\bigl[ \lvert Y - m\rvert^3 \bigr] \Biggr\}^{1/4}.
\end{multline*}
Let us now use the fact that $\B{E}\bigl[ (Y-m)^4 \bigr] \leq \kappa v^2$ 
to deduce that
$$
\B{E} \bigl[ \lvert Y - m \rvert^3 \bigr] 
\leq \sqrt{\B{E} \bigl[ (Y-m)^2 \bigr] \B{E} \bigl[ 
(Y-m)^4 \bigr]} \leq \sqrt{\kappa v^3}.
$$
Let us set $\ds \alpha = \sqrt{\frac{2 \log(\epsilon^{-1})}{nv}}$,
for this value of $\alpha$
\begin{multline*}
\biggl\lvert \frac{1}{n} \sum_{i=1}^n \bigl[ G_i -  
\B{E}(G) \bigr] \biggr\rvert \leq 
\bigl( \epsilon n^3 \bigr)^{-1/4} \Biggl\{ \frac{3^5}{10^4} (n-1) 
\kappa^2 \alpha^8 v^4 + \kappa \alpha^4 v^2 + \frac{6}{25} 
\kappa^{3/2} \alpha^7 v^{7/2} \Biggl\}^{1/4}
\\ \leq \bigl(\epsilon n^3 \bigr)^{-1/4} 
\alpha \sqrt{v} \Biggl\{\kappa + \frac{ 3^5 (n-1) \log(\epsilon^{-1})^2}{2500 n^2} \kappa^2 + \frac{12 \sqrt{2} \log(\epsilon^{-1})^{3/2}}{25 
n^{3/2}} \kappa^{3/2} \Biggr\}^{1/4} 
\end{multline*}

Let us remark also that
$$
\lvert \B{E} (G) \rvert \leq \frac{\alpha^3}{5} \B{E} \bigl( 
\lvert Y - m \rvert^3 \bigr) 
\leq \frac{2 \alpha \log(\epsilon^{-1}) \sqrt{\kappa v}}{5 n}.
$$

Putting all this together, we see that with probability 
at least $1 - 3 \epsilon$, 
\begin{multline*}
\frac{\lvert M - m \rvert }{\sqrt{v}} \leq \sqrt{\frac{2 \log(\epsilon^{-1})}{n}} 
+ \frac{ 2 \log(\epsilon^{-1}) \sqrt{\kappa}}{5n}
\\ + \Biggl( \frac{\kappa}{\epsilon n^3} \Biggr)^{1/4}  
\Biggl( 1 + \frac{3^5(n-1) \log(\epsilon^{-1})^2 \kappa}{2500 n^2} 
+ \frac{12 \sqrt{2} \log(\epsilon^{-1})^{3/2} \sqrt{\kappa}}{25n^{3/2}} 
\Biggr)^{1/4}.
\end{multline*}
The result stated in Proposition \thmref{prop6.1} is then obtained by replacing
$\epsilon$ with $\frac{2}{3} \epsilon$, to get an event with confidence 
level $1 - 2 \epsilon$ as elsewhere in this paper. 

Let us remark that the following proposition, based on the Chebyshev 
inequality applied directly to the fourth moment of the empirical 
mean does not provide the right speed when $\epsilon$ is small 
and $n$ large.

\begin{prop}
For any probability distribution whose kurtosis is not greater than $\kappa$, 
the empirical mean $M$ is such that with probability at least $1 - 
2 \epsilon$, 
$$
\lvert M - m \rvert \leq \biggl( \frac{3(n-1) + \kappa}{2 n \epsilon} 
\biggr)^{1/4} \sqrt{\frac{v}{n}}
$$
\end{prop}

\begin{proof} Let us assume to simplify notations and without
loss of generality that $\B{E}(Y) = 0$. 
$$
\B{E} \bigl( M^4 \bigr) = 
\frac{1}{n^4} \sum_{i=1}^n \B{E} (Y_i^4) 
+ \frac{1}{n^4} \sum_{i < j} 6 \B{E}(Y_i^2)\B{E}(Y_j^2) 
= \frac{\B{E}(Y^4)}{n^3} + \frac{3 (n-1) \B{E}(Y^2)^2}{n^3}.
$$
It implies that 
$$
\B{P} \Bigl( \lvert M - m \rvert \geq \eta \Bigr) \leq \frac{\B{E}(M^4)}{\eta^4} 
\leq \frac{\bigl[3(n-1) + \kappa \bigr] v^2}{n^3 \eta^4},
$$
and the result is proved by considering 
$\ds 2 \epsilon = \frac{\bigl[ 3(n-1)+\kappa \bigr] v^2}{n^3 \eta^4}$.  
\end{proof}

\subsection{Proof of Proposition \thmref{prop2.1.2}}

Let us consider the distributions $\B{P}_1$ and 
$\B{P}_2$ of the sample $(Y_i)_{i=1}^n$ obtained when 
the marginal distributions are respectively the Gaussian
measure with variance $v$ and mean $m_1 = -\eta$ and the 
Gaussian measure with variance $v$ and mean $m_2 = \eta$. 
We see that, whatever the estimator $\wh{\theta}$, 
\begin{multline*}
\B{P}_1(\wh{\theta} \geq m_1 + \eta) 
+ \B{P}_2(\wh{\theta} \leq m_2 - \eta) = 
\B{P}_1(\wh{\theta} \geq 0) + \B{P}_2(\wh{\theta} \leq 0)
\\ \geq (\B{P}_1 \wedge \B{P}_2)(\wh{\theta} \geq 0) 
+ (\B{P}_1 \wedge \B{P}_2) (\wh{\theta} \leq 0) 
= \lvert \B{P}_1 \wedge \B{P}_2 \rvert,
\end{multline*}
where $\B{P}_1 \wedge \B{P}_2$ is the measure whose density
with respect to the Lebesgue measure (or equivalently 
with respect to any dominating measure, such as $\B{P}_1 + 
\B{P}_2$) is the minimum of the 
densities of $\B{P}_1$ and $\B{P}_2$ and whose total variation 
is $\lvert \B{P}_1 \wedge \B{P}_2 \rvert$. 

Now, using the fact that the empirical mean is a sufficient 
statistics of the Gaussian shift model, it is easy to 
realize that 
$$
\lvert \B{P}_1 \wedge \B{P}_2 \rvert = \B{P}_1(M \geq m_1 + \eta) 
+ \B{P}_2(M \leq m_2 - \eta),
$$
which obviously proves the proposition. 

\subsection{Proof of Proposition \thmref{prop2.2}}

Let us consider the distribution with support $ \{ 
- n \eta, 0, n \eta \}$ defined by
$$
\B{P}\bigl( \{n \eta\} \bigr) = \B{P}\bigl( \{-n\eta \}\bigr) =  
\bigl[ 1 - \B{P}(\{0\}) \bigr]/2 = \frac{v}{2n^2 \eta^2}.
$$
It satisfies $\B{E}(Y) = 0$, $\B{E}(Y^2) = v$ and 
$$
\B{P}(M \geq \eta) = \B{P}(M \leq - \eta) 
\geq \B{P}(M = \eta) = \frac{v}{2n\eta^2}
\left(1 - \frac{v}{n^2\eta^2} \right)^{n-1}.
$$

\subsection{Proof of Proposition \thmref{prop7.3}}

Let us consider for $Y$ the following distribution, 
with support $\{-n\eta, - \xi, \xi, n \eta \}$, where $\xi$ and $\eta$ are two positive real parameters, 
to be adjusted to obtain the desired variance and kurtosis.

\begin{align*}
\B{P}(Y = - n \eta) & = \B{P}(Y = n \eta) = q,\\
\B{P}(Y = - \xi) & = \B{P}(Y = \xi ) = \frac{1}{2} - q.
\end{align*}

In this case
\begin{align*}
m & = 0,\\
\B{E}(Y^2) & = v = (1 - 2q) \xi^2 + 2 q n^2 \eta^2, \\ 
\B{E}(Y^3) & = 0,\\
\B{E}(Y^4) & = (1 - 2q) \xi^4 + 2 q n^4 \eta^4.
\end{align*}
Let us choose $\xi$ such that $v = 1$. 
This is done by putting
$$
\xi^2 = \frac{1 - 2 q n^2 \eta^2}{1 - 2 q}.
$$

The kurtosis of the distribution defined by $q$ and $\eta$, 
the two remaining free parameters once $\xi$ has been 
set as explained, is equal to 
$$
\kappa = \B{E} \bigl( Y^4 \bigr) 
= \frac{ \bigl( 1 - 2q n^2 \eta^2 \bigr)^2}{1 - 2q} + 2 q n^4 
\eta^4.
$$

It is easily seen that 
$$
\B{P} \bigl( M \geq \eta \bigr) = 
\B{P} \bigl( M \leq - \eta \bigr) 
\geq n q \frac{(1 - 2q)^{n-1}}{2} = \epsilon.
$$
Indeed, 
\begin{multline*}
\B{P} \bigl( M \geq \eta \bigr) \geq \sum_{i=1}^n \B{P} 
\biggl( Y_i = n \eta;\; Y_j \in \{-\xi,+\xi\}, j \neq i;\; \sum_{j, j \neq i} 
Y_j \geq 0 \biggr)
\\ =  n \B{P} (Y_1 = n \eta) \B{P} \biggl( Y_j \in \{-\xi, +\xi\},
j = 2, \dots, n; \; \sum_{j=2}^n Y_j \geq 0 \biggr) 
\\ \geq \frac{n q}{2}  \B{P} \Bigl( Y_j \in \{-\xi, +\xi\}, 
j=2, \dots, n \Bigr)
= \frac{nq}{2} (1 - 2 q)^{n-1}.
\end{multline*}

Starting from $\epsilon \leq (4 e )^{-1}$, and 
$c \geq 1 + 1/n$, we can define a probability distribution by choosing
\begin{align*}
q & = \frac{2 \epsilon}{n} \biggl( 1 - \frac{4e\epsilon}{n} \biggr)^{-(n-1)} 
\leq \frac{2 e \epsilon}{n} \leq \frac{1}{2n}, \\
\eta & = \biggl( \frac{c-1}{2 q n^4} \biggr)^{1/4} 
= \biggl( \frac{c-1}{4 \epsilon n^3} \biggr)^{1/4} 
\biggl( 1 - \frac{4 e \epsilon}{n} \biggr)^{(n-1)/4} \geq \frac{1}{n},
\end{align*}
whose kurtosis $\kappa$ will not be greater than $c$, since
in this case
$$
\kappa = \frac{(1 - 2qn^2 \eta^2)^2}{1 - 2q} + 2 q n^4 \eta^4 
\leq 1 + 2qn^4 \eta^4 \leq c, 
$$ 
and for which 
$$
\B{P} \Bigl( \lvert M - m \rvert \geq \eta \Bigr) 
\geq nq (1 - 2 q)^{n-1} \geq nq \biggl( 1 - \frac{4 e \epsilon}{n} 
\biggr)^{n-1} = 2 \epsilon. 
$$

\section{Generalizations}

\subsection{Non identically distributed independent random 
variables} 

The assumption that the sample is identically distributed can 
be dropped. Indeed, assuming only that the random variables
$(Y_i)_{i=1}^n$ are independent, meaning that their 
joint distribution is of the product form $\bigotimes_{i=1}^n 
\B{P}_i$, we can still write, 
for $W_i = \pm \alpha (Y_i - \theta)$ or $W_i = \pm 
\bigl[ \alpha(Y_i - \theta) - \delta \bigr]$, 
\begin{multline*}
\B{E} \biggl\{ \exp \biggl[ \sum_{i=1}^n \log \biggl( 
1 + W_i + \frac{W_i^2}{2} 
\biggr) \biggr] \biggr\}\\ 
= \exp \biggl\{ \sum_{i=1}^n \log \biggl[ 1 + \B{E} \bigl( W_i) 
+ \frac{ \B{E}\bigl( W_i^2 \bigr)}{2} \biggr] \biggr\} 
\\ \leq  \exp \biggl\{ n \log \biggl[ 1 + \frac{1}{n} \sum_{i=1}^n 
\B{E} \bigl( W_i \bigr) + \frac{1}{2n} \sum_{i=1}^n 
\B{E} \bigl( W_i^2 \bigr) \biggr] \biggr\}.
\end{multline*}

Starting from these exponential inequalities, we can reach the same 
conclusions as in the i.i.d. case, as long as we set
\begin{align*}
m & = \frac{1}{n} \sum_{i=1}^n \B{E}(Y_i),\\
\text{and } \quad v & = \frac{1}{n} \sum_{i=1}^n \B{E} \bigl[ (Y_i - m)^2 \bigr].
\end{align*}
Thus here, the role that is played by the marginal sample distribution 
in the i.i.d. case is played by the mean marginal sample distribution 
$\ds \frac{1}{n} \sum_{i=1}^n \B{P}_i$.
As moreover, the empirical mean $M$ still satisfies 
$$
\B{E} \bigl[ (M - m)^2 \bigr] = \frac{v}{n} 
- \frac{1}{n^2} \sum_{i=1}^n \bigl[ \B{E}(Y_i) - m \bigr]^2 
\leq \frac{v}{n},
$$
we see that Propositions \thmref{prop3.1}, \thmref{prop3.2},
\thmref{prop2.1}, \thmref{prop3.1.2}, \thmref{prop1.4},
and \thmref{prop5.2} remain true, the proofs being unchanged, 
except for the starting inequalities mentioned above, 
and the kurtosis coefficients being those of the mean sample 
distribution $\ds \frac{1}{n} \sum_{i=1}^n \B{P}_i$.

\subsection{Simpler truncating function}

We can use the simpler truncation function
$$
L(x) = \max \Bigl\{ -1, \min \{ +1, x \} \Bigr\}.
$$
Let $\lambda$ be the positive root of 
the equation 
$$
- \frac{1}{\lambda} \log \biggl( 1 - \frac{\lambda^2}{4 
\bigl[ \exp(\lambda) - 1 - \lambda \bigr]} \biggr) = 1.
$$
Let us define the upper and lower bounds 
\begin{align*}
L_+(x) & = \frac{1}{\lambda} \log \Bigl\{ 1 + \lambda x + 
\bigl[\exp(\lambda)-1-\lambda\bigr] x^2 \Bigr\},\\
L_+'(x) & = \frac{1}{\log(2)} \log \biggl\{ 
1 + \log(2) x + \frac{\log(2)^2}{2} x^2 + \Bigl[ 
1 - \log(2) - \frac{\log(2)^2}{2} \Bigr] x_+^3 \biggr\},\\
L_-(x) & = - L_+(-x),\\
L_-'(x) & = - L_+'(-x). 
\end{align*}
Numerically, $0.535 \leq \lambda \leq 0.536$. 
Moreover $\ds \exp(\lambda) - 1 - \lambda = \frac{ a \lambda^2}{2}$, 
with $\ds a = \frac{2 \bigl[ \exp(\lambda) - 1 - \lambda \bigr]}{\lambda^2} 
\simeq 1.2$.  
\begin{lemma}
They are such that 
\begin{align}
\label{eq10.1}
L_-(x) \leq L(x) & \leq L_+(x), \qquad x \in \B{R},\\
\label{eq10.2}
L_-'(x) \leq L(x) & \leq L_+'(x), \qquad x \in \B{R}.
\end{align}
\end{lemma}
\begin{proof}
Let us consider the function $f(x) = \exp \bigl[ \lambda L_+(x) \bigr] 
- \exp \bigl[ \lambda L(x) \bigr]$. It is such that
\begin{align*}
f(x) & = 1 + \lambda x + \frac{a \lambda^2 x^2}{2} 
- \exp \bigl[ \lambda L(x) \bigr],\\
f'(x) & = \begin{cases} 
\lambda + a \lambda^2 x - \lambda \exp \bigl(\lambda x), & x 
\in )-1, +1(,\\
\lambda + a \lambda^2 x, & x \not\in (-1, +1),
\end{cases} \\
f''(x) & = 
\begin{cases}
\lambda^2\bigl[ 1 - \exp(\lambda x) \bigr], & x \in )-1, +1(,\\
\lambda^2 a, & x \not\in (-1,+1).
\end{cases}
\end{align*}
Since $f'(0) = 0$ and $\ds f''(x) \geq 0, -1 < x \leq \frac{\log(a)}{\lambda}$, 
$f''(x) \leq 0$, $\ds \frac{\log(a)}{\lambda} \leq x < 1$,  
and $f(1) = 0$, we see that $f(x) \geq x$, $x \in (-1,+1)$. 
Moreover, $f$ is quadratic on $)-\infty, -1)$ and reach its minimum
at point $x = - \frac{1}{a \lambda}$, thus it is non negative on the whole
line when this minimum value is non negative, that is when 
$\ds 1 - \frac{1}{2a} - \exp(- \lambda) \geq 0$, which is satisfied according 
to the definition of $\lambda$.
This proves that $L(x) \leq L_+(x)$, $x \in \B{R}$. The fact that 
$L_-(x) \leq L(x)$ is then a consequence of $L(x) = -L(-x)$. 
The proof of $L(x) \leq L_+'(x)$ is done similarly by analyzing the 
shape of the function $x \mapsto \exp \bigl[ \log(2) L_+'(x) 
\bigr] - \exp \bigl[ \log(2) L(x) \bigr]$.
\end{proof}\\[-3ex]
\mbox{} \hfill \includegraphics{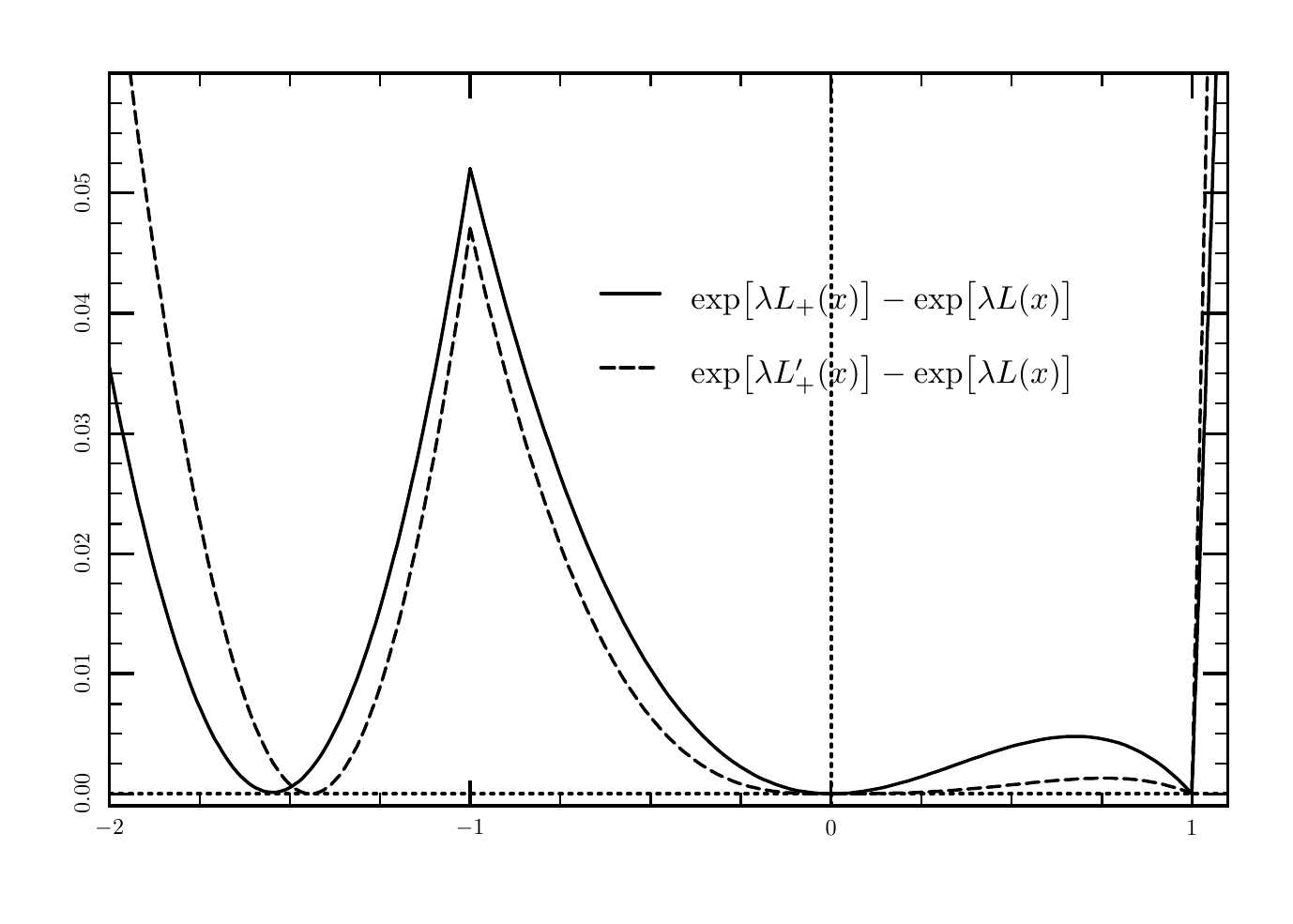} \hfill \mbox{}\\
\mbox{} \hfill 
\raisebox{-8ex}[0.7\height][0ex]{\includegraphics{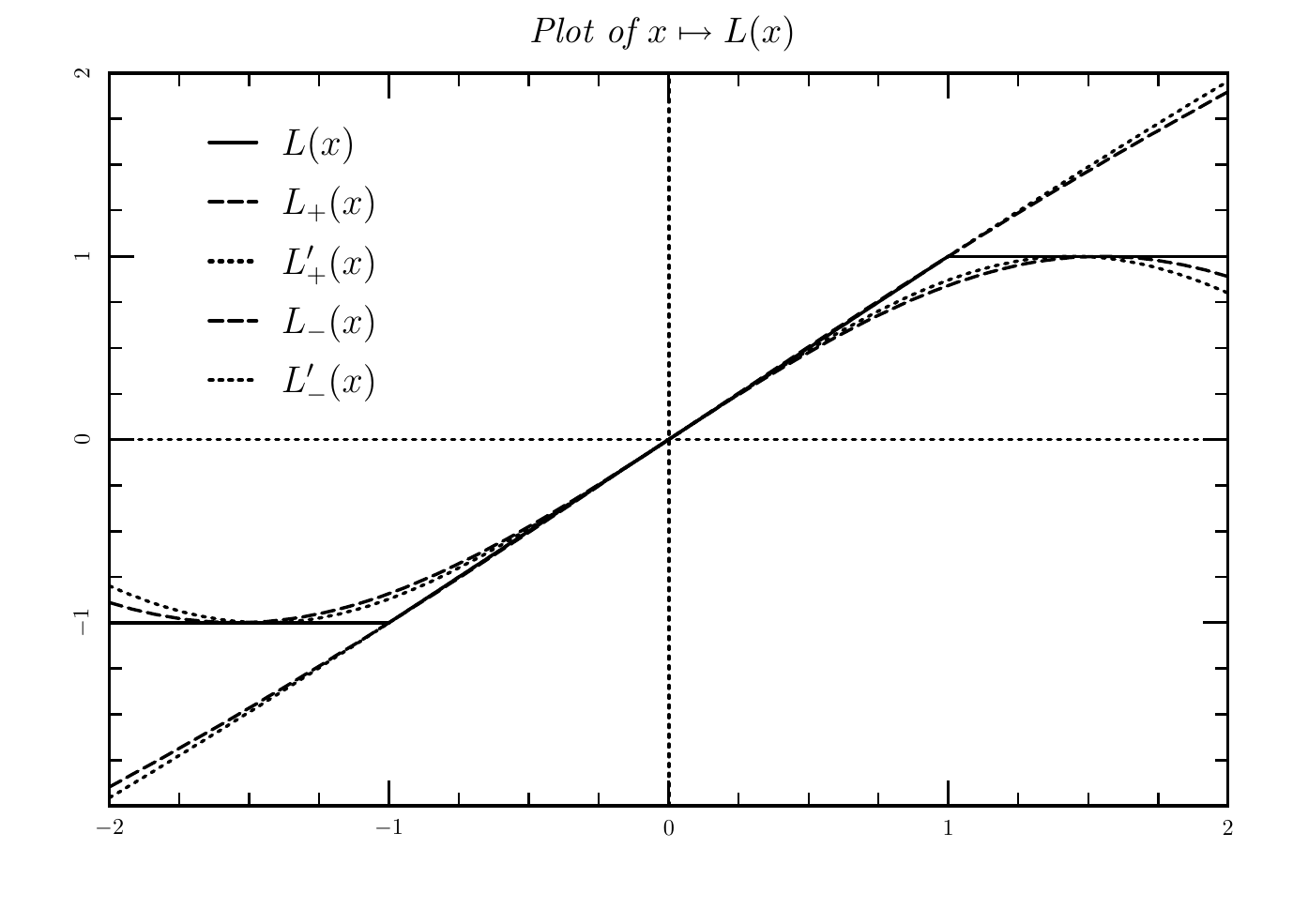}}
\hfill \mbox{}\\[-2ex]
If we redefine now our truncated mean estimate as 
$$
\wh{\theta}_{\alpha}(\theta_0) = \theta_0 + \frac{\lambda}{n \alpha} 
\sum_{i=1}^n L\biggl[ \frac{\alpha}{\lambda} (Y_i - \theta_0) \biggr],
$$
and define $\ds a = \frac{2 \bigl[ \exp(\lambda) - 1 - \lambda 
\bigr]}{\lambda^2} \simeq 1.2$, we deduce that
\begin{multline*}
\B{E} \Bigl\{ \exp \Bigl[ n \alpha  \bigl[ 
\wh{\theta}_{\alpha}(\theta_0) - \theta_0 \bigr] \Bigr]
\Bigr\} \leq \exp \Bigl\{ n \log \Bigl[ 
1 + \alpha (m - \theta_0) \\ + 
\frac{a \alpha^2}{2} \bigl[ v + (m - \theta_0)^2 \bigr] \Bigr] \Bigr\}.
\end{multline*}
Therefore, with probability at least $1 - \epsilon$,
\begin{multline*}
\wh{\theta}_{\alpha}(\theta_0) \leq  \theta_0 + 
\frac{1}{\alpha}
\log \Bigl\{ 1 + \alpha(m - \theta_0) 
+  \frac{a \alpha^2 }{2}
\bigl[ v + (m- \theta_0)^2 \bigr] \Bigr\} + \frac{\log(\epsilon^{-1})}{n 
\alpha} 
\\ \leq 
m + \frac{a \alpha  
\bigl[ v + (m - \theta_0)^2 \bigr]}{2 } + 
\frac{\log(\epsilon^{-1})}{n \alpha}.
\end{multline*}
Working out the reverse inequality in the same way gives the following 
variant of Proposition \thmref{prop3.2}.
\begin{prop}
Assume that $v \leq v_0$ and $\lvert m - \theta_0 \rvert \leq \delta_0$, 
where $v_0$ and $\delta_0$ are known prior bounds. 
With probability at least $1 - 2 \epsilon$, 
$$
\lvert \wh{\theta}_{\alpha}(\theta_0) - m \rvert 
\leq \frac{ a \alpha (v_0 + \delta_0^2)}{2} + \frac{\log(\epsilon^{-1})}{n 
\alpha}.
$$
When $\ds \alpha = \sqrt{ \frac{2 \log(\epsilon^{-1})}{
a (v_0 + \delta_0^2) n}}$, 
we get with probability at least $1 - 2 \epsilon$, 
$$
\bigl\lvert \wh{\theta}_{\alpha}(\theta_0) - m \bigr\rvert
\leq \sqrt{\frac{ 2 a  (v_0 + \delta_0^2) 
\log(\epsilon^{-1})}{n}} 
\leq 1.1 \sqrt{\frac{2 (v_0 + \delta_0^2) \log(\epsilon^{-1})}{n}}.
$$
\end{prop} 
So there is a ten per cent loss of accuracy with respect to Proposition 
\thmref{prop3.2}: this is the price to pay for using 
a simpler truncation function. We let the reader derive 
by himself the equivalent of 
the iterated estimate of Proposition \thmref{prop2.1}.
When there is a third moment, we can also work with 
Equation \myeq{eq10.2}, instead of Equation 
\myeq{eq10.1}.

\section{Some concluding remarks}

We would like to end this paper by sharing some guess about what is
going on behind the scene.
The need for thresholding indicates that large values 
may not be reliable,
and have, so to speak, a bad ``signal to noise ratio''.
This is somehow understandable, since values whose deviation 
from the mean is much larger than the standard deviation  
have to appear in the sample with a small and therefore
hard to estimate probability
whereas their large size gives them a strong impact on the empirical mean, 
which makes their contributions to this estimate even worse.

\bibliographystyle{plain}
\bibliography{ref}

\end{document}